\documentclass[11pt,reqno]{amsart}
\usepackage{tikz}
\usetikzlibrary{shapes.geometric}
\usepackage[center]{caption}
\usepackage{xcolor}

\usepackage{amssymb,latexsym,amsmath}
\usepackage{graphicx,mathrsfs}
\usepackage{subfigure}
\usepackage{float}
\usepackage{hyperref,url}
\hypersetup{colorlinks=true, urlcolor=black, linkcolor=black, citecolor=black}
\usepackage{pict2e}
\usepackage{amstext}
\usepackage{bbm}
\usepackage{amsthm}
\usepackage{dsfont}
\usepackage{amsfonts}

\newtheorem{theorem}{Theorem}[section]
\newtheorem{lemma}[theorem]{Lemma}

\newtheorem{proposition}[theorem]{Proposition}
\newtheorem{remark}[theorem]{Remark}

\newtheorem{definition}[theorem]{Definition}
\newtheorem{corollary}[theorem]{Corollary}

\newcommand{\n}{M}
\newcommand{\twoM}{M}
\newcommand{\rk}{j}
\newcommand{\jk}{j}
\newcommand{\zero}{1}
\newcommand{\one}{2}
\newcommand{\adp}{\mu}
\newcommand{\EE}{\mathbb{E}}
\newcommand{\FF}{\mathbb{F}}
\newcommand{\Eb}{\mathbf{E}}
\newcommand{\Fb}{\mathbf{F}}
\newcommand{\Pb}{\mathbf{P}}
\newcommand{\lE}{\ell} 
\newcommand{\ConstOne}{\frac{1}{24}}
\newcommand{\ConstTwo}{8}

\newcommand{\interior}{\mathrm{int}}
\newcommand{\bdry}{\mathrm{bdry}}
\newcommand{\closure}{\mathrm{cl}}
\newcommand{\convexhull}{\mathrm{co}}
\newcommand{\convex}{\mathrm{cov}}

\newcommand{\V}{T} 

\newcommand{\eq}[2]{\begin{equation} \label{#1} #2 \end{equation}}

\numberwithin{equation}{section}
\begin{document}

\title[Hausdorff dimension of Besicovitch sets]
{On the generalized Hausdorff dimension of Besicovitch sets}
\author{Xianghong Chen}
\author{Lixin Yan}
\author{Yue Zhong$^\ast$}
\address{Xianghong Chen, Department of Mathematics, Sun Yat-sen University, Guangzhou, 510275, P.R. China}
\email{chenxiangh@mail.sysu.edu.cn}
\address{Lixin Yan, Department of Mathematics, Sun Yat-sen University, Guangzhou, 510275, P.R. China}
\email{mcsylx@mail.sysu.edu.cn}
\address{Yue Zhong, Department of Mathematics, Sun Yat-sen University, Guangzhou, 510275, P.R. China}
\email{zhongy69@mail3.sysu.edu.cn}
\thanks{$^{\ast}$Corresponding author}
\dedicatory{} 
\commby{} 


\begin{abstract}
Keich 
\cite{Keich1999} 
showed that the sharp gauge function 
for the generalized Hausdorff dimension 
of Besicovitch sets 
in $\mathbb R^2$ 
is between 
$r^2\log 1/r$ and 
$r^2(\log 1/r) (\log\log 1/r)^{2+\varepsilon}$,
by refining an argument of
Bourgain \cite{Bourgain1991}. 
It is not known whether 
the iterated logarithms 
in Keich's bound 
are necessary. 
In this paper 
we construct a family of 
Besicovitch line sets 
whose sharp gauge function 
is smaller than $r^2(\log 1/r) (\log\log 1/r)^{\varepsilon}$. 
Moreover, 
these Besicovitch sets 
are minimal in the sense that 
there is essentially only 
one line in the set 
pointing in each direction. 
\end{abstract}

\maketitle

\tableofcontents


\section{Introduction}
\label{sec:intro}
\vspace{1em}

A compact set $E\subset\mathbb R^2$ is called a \textit{Besicovitch set} 
if it has zero Lebesgue measure and it contains a unit line segment in every direction. 
The notion of Besicovitch set is closely related to that of \textit{Kakeya set}, a set in which a unit line segment can be rotated continuously through 180$^\circ$. 
After Besicovitch's original construction \cite{Besicovitch1919}, \cite{Besicovitch1928},  
alternative constructions of Besicovitch set were found by Perron \cite{Perron1928}, 
Besicovitch \cite{Besicovitch1964},
Sawyer \cite{Sawyer1987}, 
K\"orner \cite{Koerner2003}, 
Babichenko et al. \cite{Babichenko2014} 
using different approaches
(see also \cite{Rademacher1962}, \cite{Schoenberg1962}, \cite{Kahane1969}, \cite{Alexander1975}, \cite{Keich1999}, \cite{Wolff1999}). 
We refer the reader to 
Falconer \cite[\S\,12.1]{Falconer2003}, Mattila \cite[Ch.\,18]{Mattila1995}, \cite[Ch.\,22]{Mattila2015}, and Bishop-Peres \cite[Ch.\,9]{Bishop2017} for more background on Besicovitch set, and to 
Wolff \cite{Wolff1999} and 
Tao \cite{Tao2001} 
for its connections to harmonic analysis. 

It is well-known that every Besicovitch set has Hausdorff dimension 2 \cite{Davies1971}. 
In this paper, 
we consider the refined problem of measuring Besicovitch\linebreak sets in terms of \textit{generalized Hausdorff dimension} (cf. \cite[Ch.\,4]{Mattila1995}, \cite[\S\,2.5]{Falconer2003}). 
Based on earlier work by Cordoba \cite{Cordoba1977}, 
Keich \cite{Keich1999} showed that the 
sharp gauge function 
for the Minkowski dimension of Besicovitch sets is 
$r^2\log(1/r)$. 
By refining an argument of Bourgain \cite{Bourgain1991}, 
he further showed that 
the sharp gauge function for the generalized Hausdorff dimension of Besicovitch sets is between $r^2\log(1/r)$ and 
$r^2\log(1/r)({\log\log}(1/r))^{2+\varepsilon}$. 
In particular, he left open 
the question whether a Besicovitch set can have zero Hausdorff measure with respect to the gauge function $r^2\log(1/r)$. 
In this paper we consider a natural family of Besicovitch sets (denoted by $\mathcal K$ below), 
and show that within this family 
one can find explicit examples of Besicovitch set whose sharp gauge function is arbitrarily close to $r^2 \log(1/r)$. 

Our main result can be stated as follows. \\[-.5em]

\noindent\textbf{Theorem A}.  
\textit{There exists a Besicovitch set $E\in\mathcal K$ whose Hausdorff measure with respect to}
$r^2\,{\log}(1/r)$
\textit{is finite, and whose 
Hausdorff measure with respect to}
$r^2\, {\log}(1/r)\,{(\log\log}(1/r))^\varepsilon$ \textit{is infinite for all $\varepsilon>0$}.\\[-.5em]

Indeed, the iterated logarithm in Theorem A can be replaced by any factor that grows to infinity as $r\rightarrow 0$; see Theorem \ref{thm:Th1} below for the precise statement. 

The Besicovitch sets in $\mathcal K$ are 
constructed as \textit{line sets} (cf. \cite{Besicovitch1964}), that is, each $E\in\mathcal K$ can be written as   
$$E=\bigcup_{l\in\mathcal L} l$$
for a family $\mathcal L$ of line segments. 
We also show that each Besicovitch set $E\in\mathcal K$ is `\textit{minimal}' in the sense that there is only one line segment contained in $E$ pointing in each direction, 
except for countably many directions for which there are two such line segments.\\[-.5em]

\noindent\textbf{Theorem B}. 
\textit{For each $E\in\mathcal K$, there is a countable set $\Omega\subset\mathbb S^1$ of directions so that}

\noindent
\textit{\emph{(}i\emph{)} if $\omega\in\mathbb S^1\backslash\Omega$, there is exactly one line $\ell\parallel\omega$ with $\mathcal H^1(\ell\cap E)>0$};

\noindent\textit{\emph{(}ii\emph{)} if $\omega\in\Omega$, there are exactly two lines $\ell_i\parallel\omega$ $(i=1,2)$ with $\mathcal H^1(\ell_i\cap E)>0$.}\\[-.5em]

Here as usual $\mathcal H^1$ denotes 
the linear Hausdorff measure 
in $\mathbb R^2$. 
Indeed, the exceptional directions $\Omega$ and the lines 
$\ell, \ell_i$ in Theorem B
can be described rather precisely. 
See Theorem \ref{thm:Th2} for details. 
We remark that 
K\"orner \cite{Koerner2003} 
and K\'atay \cite{Katay2020} 
have shown that 
stronger `minimality' properties hold for 
typical (in the sense of Baire category) 
Besicovitch sets 
without exceptional directions. 

By rotation, it suffices to consider compact sets which contain unit line segments with slope ranging from 0 to 1. 
We shall continue to call such a set a Besicovitch set and identify $E\in\mathcal K$ with such sets. 

Similar to Besicovitch's construction \cite{Besicovitch1928}, 
the sets $E\in\mathcal K$ are constructed based on two elements: 
(i) a Kakeya-type set construction and 
(ii) iteration. 
Given a dyadic integer $M=2^m$, 
a variant of Besicovitch's `cut-and-slide' procedure 
is used to generate a Kakeya-type set $\mathbf E_M$, which 
is a union of $2^M$ many triangles with base $2^{-M}$, height $1$, and 
pointing in $2^{-M}$-separated directions,  so that 
$$|\mathbf E_M|\approx \frac1M,\quad \text{as } M\rightarrow\infty.$$ 
As shown in \cite{Keich1999}, the order $\frac1M$ is smallest possible by the work of Cordoba \cite{Cordoba1977}. 
This order is achieved in the constructions by 
Perron \cite{Perron1928} 
and 
Schoenberg \cite{Schoenberg1962} 
(see also \cite{Schoenberg1962a}, \cite{Babichenko2014}, \cite[Ch.\,9]{Bishop2017}). 
The Kakeya-type set $\mathbf E_M$ used in this paper corresponds to the minimal construction in \cite{Schoenberg1962}. 

The second element of the construction is an iteration which runs the cut-and-slide procedure over each triangle in $\mathbf E_M$ (and so forth). 
More precisely, 
given a nondecreasing sequence $\{M_n\rightarrow\infty\}$, 
such an iteration 
starts with $\mathbf F_1 = \mathbf E_{M_1}$ 
and generates a sequence of Kakeya-type sets 
$\mathbf F_1, \mathbf F_2,\cdots$, whose limit set 
$$E=\lim_{n\rightarrow\infty} \mathbf F_n$$ 
becomes a Besicovitch set (denoted by $E_{\{M_n\}}$ below). 
These Besicovitch sets form the family $\mathcal K$ mentioned above, that is, 
$$\mathcal K := \big\{E_{\{M_n\}}: 
M_n \text{ is dyadic, nondecreasing, and satisfies } M_n\rightarrow\infty\big\};$$
see Section \ref{sec:F} for details. 

The proof of
Theorem A relies on choosing a suitable (rapidly growing) sequence $\{M_n\}$ 
to guarantee that 
$E_{\{M_n\}}$ has positive Hausdorff measure with respect to the gauge $r^2 \log(1/r)\log\log\log(1/r)$. 
Several geometric properties of $\mathbf E_M$ are crucial to the proof. The first one is the Perron-tree structure of $\mathbf E_M$. Although not constructed in exactly the same way as in \cite{Perron1928}, 
$\mathbf E_M$ has the similar property that, 
starting from a `trunk' of size about $\frac 1M\times\frac 1M$, 
it `branches' into $2^k$ many (essentially) solid parallelograms at height 
$\frac kM$, 
where $k=1, \cdots, M-1$. 
The second geometric property, 
which is a consequence of the first, 
is that 
$\mathbf E_M$ is $\frac 1M$-`porous' at each height, 
thus converges to a solid (up-side-down) triangle 
as $M\rightarrow\infty$. 
The third property 
has to do with the 
`mass distribution' 
of $\mathbf E_M$ 
at each height and 
within the plane, summarized in Propositions \ref{prop NI} and \ref{prop:optimum cover of EM}. 
Roughly speaking, these propositions quantify how far $\mathbf E_M$ deviates from being $r^2\log (1/r)$-dimensional at different locations and scales. 

The plan of the paper is as follows. 
In Section \ref{sec:prelim}, we introduce some notation and give some preliminaries. 
In Section \ref{sec:E}, we present the construction of $\mathbf E_M$ and prove the geometric properties mentioned above. 
In Section \ref{sec:F}, we present the construction of $E_{\{M_n\}}$ and prove some basic properties of it. 
Theorem A is restated as Theorem \ref{thm:Th1} and proved in Section \ref{sec:Th1}. 
Theorem B and its more precise version (Theorem \ref{thm:Th2}) are proved in Section \ref{sec:Th2}. 


\section{Notation and Preliminaries}
\label{sec:prelim}

Throughout the paper, 
unless otherwise stated, 
all sets (such as squares, triangles, parallelograms, intervals) are considered to be closed. 
Two sets are considered to be disjoint 
if their interiors are disjoint. 
All sets below lie in the rectangle $[0,1]\times[-1,1]$. 

We call $x\in[0,1]$ a \textit{dyadic rational} if
$x=\frac{j}{2^m}$ 
for some integers $j$ and $m$. 
An interval $I$ is called a \textit{dyadic interval} if 
$I=[\frac{j}{2^m},\frac{j+1}{2^m}]$ 
for some integers $j$ and $m$. 
$Q\subset\mathbb R^2$ is called a \textit{dyadic square} if 
$Q=I_1\times I_2$
for some dyadic intervals $I_1$ and $I_2$.  
The sidelength of a square $Q$ will be denoted by $r(Q)$. 
Unless otherwise stated, 
\textit{all squares below are dyadic}. 

We use $\delta$ to denote \textit{dyadic scales} of the form
$\delta=2^{-m},$
where $m\ge 0$ is an integer. 
Let $E\subset\mathbb R^2$ be a compact set.  
We use $\mathcal N_\delta(E)$ to denote the minimum number of $\delta\times\delta$ squares needed to cover $E$. 
By a cover of $E$, 
we mean a collection of squares $\{Q_i\}$ such that $E\subset\bigcup_i Q_i\,;$ if each $Q_i$ satisfies 
$\delta_2\le r(Q_i)\le \delta_1,$ 
then we call $\{Q_i\}$ a $[\delta_2,\delta_1]$\textit{-cover of} $E$. 
Let $\{Q_i\}$ and $\{\widetilde Q_j\}$ be two collections of squares. 
We write $$\{Q_i\}\prec\{\widetilde Q_j\}$$ to indicate that each $Q\in\{Q_i\}$ is contained in some $\widetilde Q\in\{\widetilde Q_j\}$, in which case we say that $\{Q_i\}$ \textit{is nested in} $\{\widetilde Q_j\}$. 
The generalized Hausdorff measure of $E$ will be studied, without loss of generality, through finite covers of $E$. 
This motivates the following definitions, 
which will play an important role 
in the proof of Theorem A. 

\begin{definition}[scale-limited $h$-measure]
\label{def:h-measures}
Let $h:(0,1]\to (0,\infty)$ be a continuous function with $\lim_{r\rightarrow0} h(r)=0$, 
and let $E\subset [0,1]\times[-1,1]$. 
Given two dyadic scales $\delta_2\le\delta_1$, 
the $[\delta_2,\delta_1]$\emph{-limited} $h$-\emph{measure of} $E$ is defined by
$$
\mathcal{M}_{\delta_2}^{\delta_1}(E;h) = \min \Big\{\displaystyle\sum_{i} 
h(r(Q_i)): \{Q_i\} \text{ is a $[\delta_2,\delta_1]$-cover of $E$}\Big\}. 
$$
\noindent In particular, when $\delta_1=1$ and $\delta_2=\delta$, define 
$$
\mathcal{M}_{\delta}(E;h)=\mathcal{M}_{\delta}^{1}(E;h). 
$$
\end{definition}

Note that Definition \ref{def:h-measures} also applies to general (not necessarily compact) subsets of $[0,1]\times[-1,1]$, and that $\mathcal{M}_{\delta_2}^{\delta_1}(E;h)$ is finitely subadditive with respect to $E$. 

\begin{remark}
\label{rmk:def-2.1}
When $E$ is contained in a square $\widetilde Q$ with 
$r(\widetilde Q)=\delta_1$, $\mathcal{M}_{\delta_2}^{\delta_1}(E;h)$ can be evaluated by considering only covers $\{Q_i\}$ of $E$ that are nested in $\{\widetilde Q\}$ (otherwise one can `flip' $Q_i$     along the boundary of $\widetilde Q$ to the inside of $\widetilde Q$). 
\end{remark}

\begin{definition}[optimal cover]
\label{def:optimal-cover}
Let $h$ and $E$ be as in Definition \ref{def:h-measures}. 
A collection of squares $\{Q_i\}$ 
is called an \emph{optimal} $[\delta_2,\delta_1]$-\emph{cover of} $E$ with respect to $h$, if the minimum in the definition of 
$\mathcal{M}_{\delta_2}^{\delta_1}(E;h)$ 
is attained by $\{Q_i\}$;  
if instead $\{Q_i\}$ satisfies 
$$\sum_{i}  h(r(Q_i))
\le C\cdot \mathcal{M}_{\delta_2}^{\delta_1}(E;h)$$
for an absolute constant $C\ge1$, then 
$\{Q_i\}$ is called a 
\emph{quasi-optimal $[\delta_2,\delta_1]$-cover of} $E$ with respect to $h$. 
\end{definition}

When $\delta_1=1$ and $\delta_2=\delta$, 
we call an optimal $[\delta,1]$-cover of $E$ an \emph{optimal} $\delta$\emph{-cover of} $E$. 
Note that an optimal $[\delta_2,\delta_1]$-cover of $E$ always exists 
due to the scale limitation. 
Note also that an optimal cover necessarily consists of disjoint squares due to its optimality. 

Set 
$$\log(R)=\text{max}(\log_2 R, 1),\quad R>0.$$
In Sections \ref{sec:E} and \ref{sec:Th1} 
we shall be concerned with 
$h(r)=r^2\log(1/r)$ 
and 
$h_1(r)=r^2\log(1/r)\,\phi(r)$ 
respectively, for suitable slowly-growing functions $\phi(r)$
such as 
$\log\log\log(1/r)$. 
For notational simplicity, 
we shall denote 
$$h(Q)=h(r(Q))$$
and $h_1(Q)=h_1(r(Q))$ when $Q$ is a square. 

We use $|E|$ to denote both the two-dimensional Lebesgue measure (when $E\subset\mathbb R^2$) 
and the one-dimensional Lebesgue measure (when $E\subset\mathbb R$). 
The interior, boundary, and closure of a set $E\subset\mathbb R^2$ are denoted by $\interior(E)$, $\bdry(E)$, and $\closure(E)$ respectively. 
The \emph{vertical $\delta$-neighborhood of} $E$ is defined by 
$$E(\delta)=\bigcup_{(x,y)\in E} \{x\}\times (y-\delta,y+\delta).$$ 
The vertical cross-section of $E$ at $x$ is denoted by $E|_x$, that is, 
\begin{equation}
\label{eq:E|x}
E|_x=E\cap \big(\{x\}\times\mathbb R\big).
\end{equation}
The \textit{vertical convex hull} of $E$ is defined by 
\begin{equation}
\label{eq:cov(E)}
\convex(E)=\bigcup_x\,\convexhull(E|_x),
\end{equation}
where $\convexhull(E|_x)$ stands for the usual convex hull in $\mathbb R^2$. 

A set $\mathbf E\subset\mathbb R^2$ is called \textit{solid} if for every $z\in\mathbf E$ and $\varepsilon>0$, we have 
$$|D(z,\varepsilon)\cap\mathbf E|>0.$$ 
Here as usual $D(z,\varepsilon)$ stands for the two-dimensional disk centered at $z$, of radius $\varepsilon$. 
Solid sets will often be denoted in boldface. 

\begin{lemma}[nesting of optimal covers]
\label{lem:nesting}
Let $h$ be as in Definition \ref{def:h-measures}. 
Let $\mathbf F_2\subset \mathbf F_1$ be two solid sets in $[0,1]\times[-1,1]$, 
and let $\delta_2\le \delta_1$. If $\{\widetilde Q_j\}$ is an optimal $\delta_1$-cover of $\mathbf F_1$ with respect to $h$, then there exists an optimal $\delta_2$-cover of $\mathbf F_2$ with respect to $h$ that is nested in $\{\widetilde Q_j\}$.  
\end{lemma}
\begin{proof}
Among the (finitely many) optimal $\delta_2$-covers of $\mathbf F_2$, 
let $\{Q_i\}$ be a minimal one which consists of the largest (in lexicographical order) numbers of small squares. 
We show that $\{Q_i\}\prec\{\widetilde Q_j\}$. 

First, notice that since $\{\widetilde Q_j\}$ is also a cover of $\mathbf F_2$, 
we have 
\begin{equation}
\mathbf F_2\subset\bigcup_{i,j}\,
\big(Q_i\cap\widetilde Q_j\big).\label{eq:QiQj}
\end{equation}
Each $Q_i\cap\widetilde Q_j$ is either empty or falls in one of the following three cases: 
\begin{equation}
\label{eq:3-cases}
Q_i\cap\widetilde Q_j
\begin{cases} 
= Q_i, & \text{if } Q_i\subset \widetilde Q_j; \\ 
= \widetilde Q_j, & \text{if } \widetilde Q_j \subsetneq Q_i; \\ 
\subset \bdry(Q_i), & \text{otherwise.} 
\end{cases}
\end{equation}
Since $\mathbf F_2$ is a solid set, 
the last case (boundary intersection) is redundant and can be removed from the right-hand side of \eqref{eq:QiQj}, so that 
\begin{equation*}
\mathbf F_2\subset
\Big(\bigcup_{i,j:\atop Q_i\subset \widetilde Q_j}\,
Q_i\Big)
\bigcup \Big(\bigcup_{i,j:\atop \widetilde Q_j\subsetneq Q_i}\,
\widetilde Q_j\Big).
\end{equation*}
To complete the proof, 
it now suffices to show that the second case in \eqref{eq:3-cases} does not happen. 

Assume to the contrary that there exists an $i_0$ such that $\widetilde Q_j \subsetneq Q_{i_0}$ for some $j$. 
Since $\{\widetilde Q_j\}$ are disjoint, we have 
$$\interior(Q_{i_0})\cap \mathbf F_1\;
\subset \bigcup_{j:\,\widetilde Q_j \subsetneq Q_{i_0}} \widetilde Q_j, $$
and thus 
$$\interior(Q_{i_0})\cap \mathbf F_2\;
\subset \bigcup_{j:\,\widetilde Q_j \subsetneq Q_{i_0}} \widetilde Q_j.$$ 
Consequently, using the solidity of $\mathbf F_2$ again, 
the collection of squares 
\begin{equation}
\label{eq:squares-Qj}
\big\{\widetilde Q_j: \widetilde Q_j \subsetneq Q_{i_0}\big\}
\end{equation}
can be used to replace $Q_{i_0}$ in $\{Q_i\}$ to form a new $\delta_2$-cover of $\mathbf F_2$. 
Denote this new cover by $\{Q_i'\}$. 

It remains to show that $\{Q_i'\}$ forms an optimal $\delta_2$-cover of $\mathbf F_2$, 
which would be a contradiction since $\{Q_i\}$ was chosen to minimal. 
Notice that $Q_{i_0}$ can be used to replace the squares \eqref{eq:squares-Qj} in the cover $\{\widetilde Q_j\}$ to form a new $\delta_1$-cover of $\mathbf F_1$. 
It follows from the optimality of $\{\widetilde Q_j\}$ that 
$$\sum_{j:\,\widetilde Q_j \subsetneq Q_{i_0}} h(\widetilde Q_j)\le h(Q_{i_0}).$$
This shows that $\{Q_i'\}$ forms an optimal $\delta_2$-cover of $\mathbf F_2$, 
and the proof of Lemma \ref{lem:nesting} is complete. 
\end{proof}

The next two lemmas will be used in Sections \ref{sec:F} and \ref{sec:Th1} respectively to simplify rescaling arguments. Let $E\subset\mathbb R^2$ 
and let $\delta>0$. 
We denote     
$$\delta.E=\{(x,\delta y): (x,y)\in E\}.$$
{
\begin{lemma}[sublinearity]
\label{lem:dilation}
Let $h$ and $E$ be as in Definition \ref{def:h-measures}, 
and let $\delta_2\le\delta_1$. 
Then, for any $\delta\le 1/2$, 
we have
$$\mathcal{M}_{\delta_2}^{\delta_1}(\delta.E;h)
\ge \delta\cdot\mathcal{M}_{\delta_2}^{\delta_1}(E;h).$$
\end{lemma}
\begin{proof}
Let $\{Q_i\}$ be an optimal $[\delta_2,\delta_1]$-cover of $\delta.E$ with respect to $h$, so that 
$$\sum_{i}  h(Q_i)
=\mathcal{M}_{\delta_2}^{\delta_1}(\delta.E;h).$$
Then 
$$E\subset\bigcup_i\,(\delta^{-1}.Q_i).$$
Notice that each $\delta^{-1}.Q_i$ can be decomposed into $\delta^{-1}$ many dyadic squares $Q_{i,j}$ 
with $r(Q_{i,j})=r(Q_i)$, $j=1,\cdots,\delta^{-1}$. 
Thus $\{Q_{i,j}\}$ forms a $[\delta_2,\delta_1]$-cover of $E$. Using this cover, we can bound 
\begin{align*}
\mathcal{M}_{\delta_2}^{\delta_1}(E;h)
&\le \sum_{i}\sum_{j=1}^{1/\delta}  h( Q_i^{(j)} )\\
&= \sum_{i} {\delta^{-1}} h( Q_i)\\
&= {\delta^{-1}} \mathcal{M}_{\delta_2}^{\delta_1}(\delta.E;h). 
\end{align*} 
Multiplying both sides by $\delta$, this proves Lemma \ref{lem:dilation}. 
\end{proof}
}

Let $E\subset\mathbb R^2$ 
and let $l(x)=ax+b$ be a line. 
We denote 
\begin{equation}
\label{eq:E+l-def}
E+l=\big\{(x,y+l(x)): (x,y)\in E\big\}.
\end{equation}

{
\begin{lemma}[affine stability]
\label{lem:affine}
Let $h$ and $E$ be as in Definition \ref{def:h-measures}, 
and let $\delta_2\le\delta_1$. 
Suppose $l(x)=ax+b$ satisfies $a\in[0,1]$ and $E+l\subset[0,1]\times[-1,1]$. 
Then
$$\mathcal{M}_{\delta_2}^{\delta_1}(E+l;h)
\ge \frac13\,\mathcal{M}_{\delta_2}^{\delta_1}(E;h).$$
\end{lemma}
\begin{proof}
The proof is similar to that of Lemma \ref{lem:dilation}. 
Let $\{Q_i\}$ be an optimal $[\delta_2,\delta_1]$-cover of $E+l$ with respect to $h$, so that 
$$\sum_{i}  h(Q_i)
=\mathcal{M}_{\delta_2}^{\delta_1}(E+l;h).$$
Then, writing $(-l)(x)=-ax-b$, we have  
$$E\subset\bigcup_i\,(Q_i-l).$$ 
Since $a\in[0,1]$, 
each parallelogram $Q_i-l$ can be covered by three dyadic squares $Q_{i,j}$ 
with $r(Q_{i,j})=r(Q_i)$, $j=1,2,3$.   
Thus $Q_{i,j}$ forms a $[\delta_2,\delta_1]$-cover of $E$. Using this cover we can bound 
\begin{align*}
\mathcal{M}_{\delta_2}^{\delta_1}(E;h)
&\le \sum_{i}\sum_{j=1}^{3}  h(Q_i^{(j)} )\\
&= \sum_{i} 3\,h( Q_i)\\
&= 3\,\mathcal{M}_{\delta_2}^{\delta_1}(E+l;h). 
\end{align*} 
Dividing both sides by 3, 
this proves Lemma \ref{lem:affine}. \end{proof}
}

Finally, we record three simple lemmas which will be useful in Sections \ref{sec:E}, \ref{sec:F}, and \ref{sec:Th1} respectively. 

\begin{lemma}
\label{lem:leb-solidity}
Let $\widetilde Q$ be a dyadic square. If $Q_i\subset \widetilde Q$
are dyadic squares with 
\begin{equation}
\label{eq:squares-density}
\Big|\bigcup_i Q_i\Big|\ge \rho\,|\widetilde Q|,
\end{equation}
then, for any decreasing function 
$\varphi:(0,1]\rightarrow (0,\infty)$,  
we have 
\begin{equation}
\label{eq:measure}
    \sum_i h(Q_i)\ge \rho\cdot h(\widetilde Q), 
\end{equation}
where $h(r)=r^2\varphi(r)$. 
\end{lemma}
\begin{proof}
By \eqref{eq:squares-density} and the subadditivity of the Lebesgue measure, we have 
\begin{equation}
\label{eq:lemma-solidity}
h(\widetilde Q) 
= \vert \widetilde Q \vert \,\varphi(r(\widetilde Q)) 
\leq \rho^{-1}\sum_i  \vert Q_i \vert \,\varphi(r(\widetilde Q)).
\end{equation}
By the monotonicity of $\varphi$, 
$$\eqref{eq:lemma-solidity}\le\rho^{-1}\sum_i  \vert Q_i \vert \,\varphi(r(Q_i))
=\rho^{-1}\sum_i  h(Q_i).$$
Multiplying throughout by $\rho$, 
this completes the proof of Lemma \ref{lem:leb-solidity}. 
\end{proof}

A simple consequence of Lemma \ref{lem:leb-solidity} is 
\eq{eq:M(Q)}{
\mathcal{M}_{\delta_2}^{\delta_1}(\widetilde Q;h)=h(\widetilde Q),
}
where $\widetilde Q$ and $h$ are as in Lemma \ref{lem:leb-solidity}, and $\delta_2\le\delta_1=r(\widetilde Q)$. 

\begin{definition}
\label{def:epsilon-dense}
Let $X=\{x_i\}_{i=1}^{N}$
be a finite set of reals with 
$x_{i}<x_{i+1}$ 
$(i=1,\cdots, N-1)$, 
and let $\varepsilon>0$. 
We say that $X$ is $\varepsilon$-\emph{dense} if 
$$\max_{1\le i<N}|x_{i+1}-x_{i}|\le \varepsilon.$$
\end{definition}

It is easy to see that $X$ is $\varepsilon$-{dense} if and only if 
$\,\bigcup_{x\in X} (I_\varepsilon+x) \text{ is connected},$
where $I_\varepsilon$ is any closed interval with $|I_\varepsilon|=\varepsilon$ 
(in particular, one can take $I_\varepsilon=[-\varepsilon,0]$ or $[0,\varepsilon]$). 

\begin{lemma}
\label{lem:dense-union}
Suppose $X, Y\subset\mathbb R$ are $\varepsilon$-dense and satisfy
\begin{equation}
\label{eq:epsilon-close}
\min_{x\in X\atop y\in Y} |x-y|\le \varepsilon.
\end{equation}
Then $X\cup Y$ is $\varepsilon$-dense. 
\end{lemma}
\begin{proof}
Denote $x_1=\min X$, $x_{N}=\max X$, $x_{N+1}=x_{N}+\varepsilon$, and $y_1=\min Y$. Without loss of generality, we may assume $x_1\le y_1$. 
Then, by \eqref{eq:epsilon-close} we must have $y_1\in [x_1,x_{N+1}]$. 
Since both $X$ and $Y$ are $\varepsilon$-dense, 
it follows by sorting $X\cup Y$ that so is $X\cup Y$.  
\end{proof}

\begin{definition}
\label{def:eta-connected}
Let $\eta>0$. A family of intervals  
$\{[x-\varepsilon-\eta,x]: x\in X\}$ is called 
$\eta$-\emph{connected} 
if $X\subset\mathbb R$ is $\varepsilon$-{dense}. 
\end{definition}

It is easy to see that $\{[x-\varepsilon-\eta,x]: x\in X\}$ is $\eta$-{connected} 
if and only if 
$\{[x,x+\varepsilon+\eta]: x\in X\}$ is $\eta$-{connected}.
A direct corollary of Lemma \ref{lem:dense-union} is the following. 

\begin{corollary}
\label{cor:eta-connected}
Let $\{I_i\}$ and $\{J_j\}$ be two $\eta$-connected interval families, with $|I_i|=|J_j|$. 
Write $I=\bigcup_i I_i$ and $J=\bigcup_j J_j$. 
If $|I\cap J|\ge \eta$, 
then the interval family $\{I_i\}\cup\{J_j\}$ is $\eta$-connected. 
\end{corollary}

The next lemma shows that, 
when one considers $[\delta,1]$-covers of a set, 
(vertical) gaps of size smaller than $\delta$ are effectively `invisible'.

\begin{lemma}
\label{lem:invisible}
Let $F\subset\mathbb R^2$. Suppose 
$F(\delta/2)$ contains a square $\widetilde Q$ with $r(\widetilde Q)\ge\delta$. Then, whenever $F$ is covered by a collection of squares $\{Q_i\}$ 
with $r(Q_i)\ge \delta,$ 
so is $\widetilde Q$.  
\end{lemma}
\begin{proof} 
The lemma is essentially a one-dimensional statement, 
since the vertical neighborhood $F(\delta/2)$ is defined section-wise. 
Suppose by contradiction 
$$R=\widetilde Q\;\backslash\bigcup_i Q_i \neq \varnothing.$$
Then, since $r(\widetilde Q), r(Q_i)\ge\delta$, 
$R$ must contain an open square $Q_\delta$ of sidelength $\delta$. 
By the assumption $\widetilde Q\subset F(\delta/2)$, 
the center of $Q_\delta\subset\widetilde Q$ is contained in $F(\delta/2)$.  
By the definition of $F(\delta/2)$, this implies
$$Q_\delta\cap F\neq\varnothing,$$ 
which is a contradiction since 
$Q_\delta$ is disjoint from 
$\{Q_i\}$ (which covers $F$). 
\end{proof}


\section{Construction of $\Eb_{\n}$}
\label{sec:E}

In this section, 
we first recall the construction of Kakeya-type sets from \cite{Schoenberg1962}. 
These Kakeya-type sets 
will be the `building blocks' 
in the iterative construction of Besicovitch sets in Section \ref{sec:F}. 
Then we prove some 
geometric and distributional properties 
of these Kakeya-type sets. 

\subsection{The construction} 
For each integer $\n\ge2$, 
we shall construct a Kakeya-type set $\Eb_{\n}$ as in  \cite{Schoenberg1962} and \cite{Keich1999} 
using a cut-and-slide procedure. 
Our exposition here is somewhat more geometrical, 
in order to facilitate the proof of several properties of $\Eb_{\n}$. 
\\ 

\noindent\textbf{Besicovitch compression for trapezoids}. 
In order to keep track of geometric configurations 
generated in the construction, 
we first apply the Besicovitch compression to 
a pair of trapezoids that share a common side,  
and introduce some notation. 
More specifically, 
given a trapezoid $T$ with vertical bases at $x=0$ and $x=b>0$ 
(Figure \ref{fig:trapezoid}, left), 
cut $T$ through the midpoints of its bases. 
We get two trapezoids $T_\zero$ (upper) and $T_\one$ (lower). 
Slide $T_\zero$ upwards as needed, 
then slide $T_\one$ upwards 
(denote by $\widetilde{T_\zero}, \widetilde{T_\one}$ the new trapezoids) 
so that the upper sides
of $\widetilde{T_\zero}$ and $\widetilde{T_\one}$ intersect at $x=c$ ($0\le c\le b$). 
We thus obtain a decomposition
\begin{equation}\label{eq:Besicovitch}
\widetilde{T_\zero}\cup \widetilde{T_\one} = T' \cup (P - A), 
\end{equation} 
where $T'$ is a new trapezoid based at $x=0$ and $x=c$; 
$P$ is a parallelogram bounded between $x=c$ and $x=b$; 
$A\subset P$ is a (half-open) triangle based at $x=b$ (Figure \ref{fig:trapezoid}, right). 
We shall often write $P=P(a)$ to indicate that the slope of the parallelogram is $a$. 
    
    \begin{figure}[ht]
    \centering
    \includegraphics[width=\linewidth]{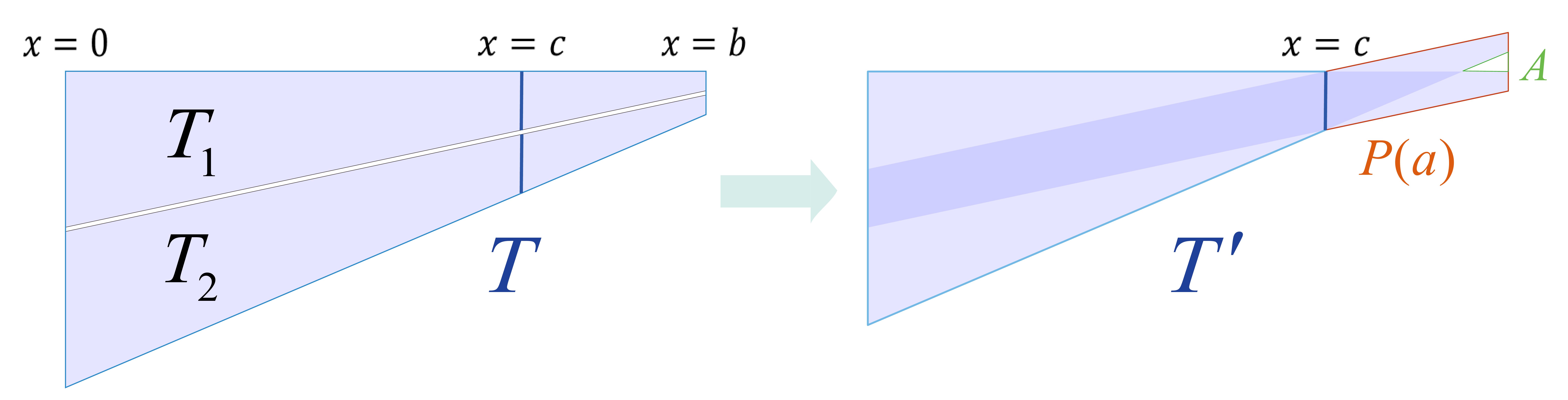}
    \caption{Besicovitch compression of $T_\zero\cup T_\one$}
    \label{fig:trapezoid}
    \end{figure}

We now describe the construction. \\

\noindent\textbf{Step 0}. 
Start with the right triangle in the fourth quadrant with vertices $(0,0)$, $(1,0)$, and $(0,-1)$ (Figure \ref{fig:E}). 
View the vertical side of the triangle as the base. 
Divide the base into $2^\n$ many intervals, each of length $2^{-\n}$. 
Using these intervals as new bases, 
and $(1,0)$ as a common vertex, 
we get $2^\n$ many thin triangles 
$S_i^{0}, i=1, \cdots, 2^\n$. 
In subsequent steps, 
we shall slide these triangles vertically 
and gradually to get $\Eb_{\n}$. 
For notational convenience, we also view $S_i^{0}$ as a trapezoid and set $T_i^{(0)}=S_i^{0}$.\\

\noindent\textbf{Step $\mathbf{\textit{\rk}}$} ($\rk=1,\cdots, \n$). 
Given $S_i^{\rk-1}$ and trapezoids $T_i^{(\rk-1)}\subset S_i^{\rk-1}$, 
$i=1, \cdots, 2^{\n-\rk+1}$, 
slide $S_{2i}^{\rk-1}$ upwards so that the upper sides of $T_{2i-1}^{(\rk-1)}$ and $T_{2i}^{(\rk-1)}$ 
intersect at 
$$x_j:=1-\frac{\rk}{\n}.$$ 
Still denote the slid copies by  ${S}_{2i}^{\rk-1}$ 
(resp. $T_{2i}^{(\rk-1)}$). 
Set $S_{i}^{\rk}=S_{2i-1}^{\rk-1}\cup S_{2i}^{\rk-1}$, 
and let $T_{i}^{(\rk)}$, $P_{i}^{\rk}$, and $A_{i}^{\rk}$ 
be defined by the decomposition 
$$T_{2i-1}^{(\rk-1)}\cup T_{2i}^{(\rk-1)}
=T_{i}^{(\rk)}\cup (P_{i}^{\rk}-A_{i}^{\rk})$$ 
described in \eqref{eq:Besicovitch}. 
Note that the trapezoids $T_{i}^{(\rk)}$ are based at $x=0$ and $x=x_j$, 
and that the parallelograms 
\begin{equation}
\label{eq:Pij-slope}
P_{i}^{\rk} = P^j\Big(\frac{2i-1}{2^{\n-\rk+1}}\Big),\quad i=1,\cdots, 2^{M-j}
\end{equation}
are bounded between $x=x_j$ and $x=x_{j-1}$. 

Finally, set $\Eb_{\n}=S_{1}^{\n}$. This completes the construction of $\Eb_{\n}$. 

    \begin{figure}[h]
    \centering    \includegraphics[width=0.95\linewidth]{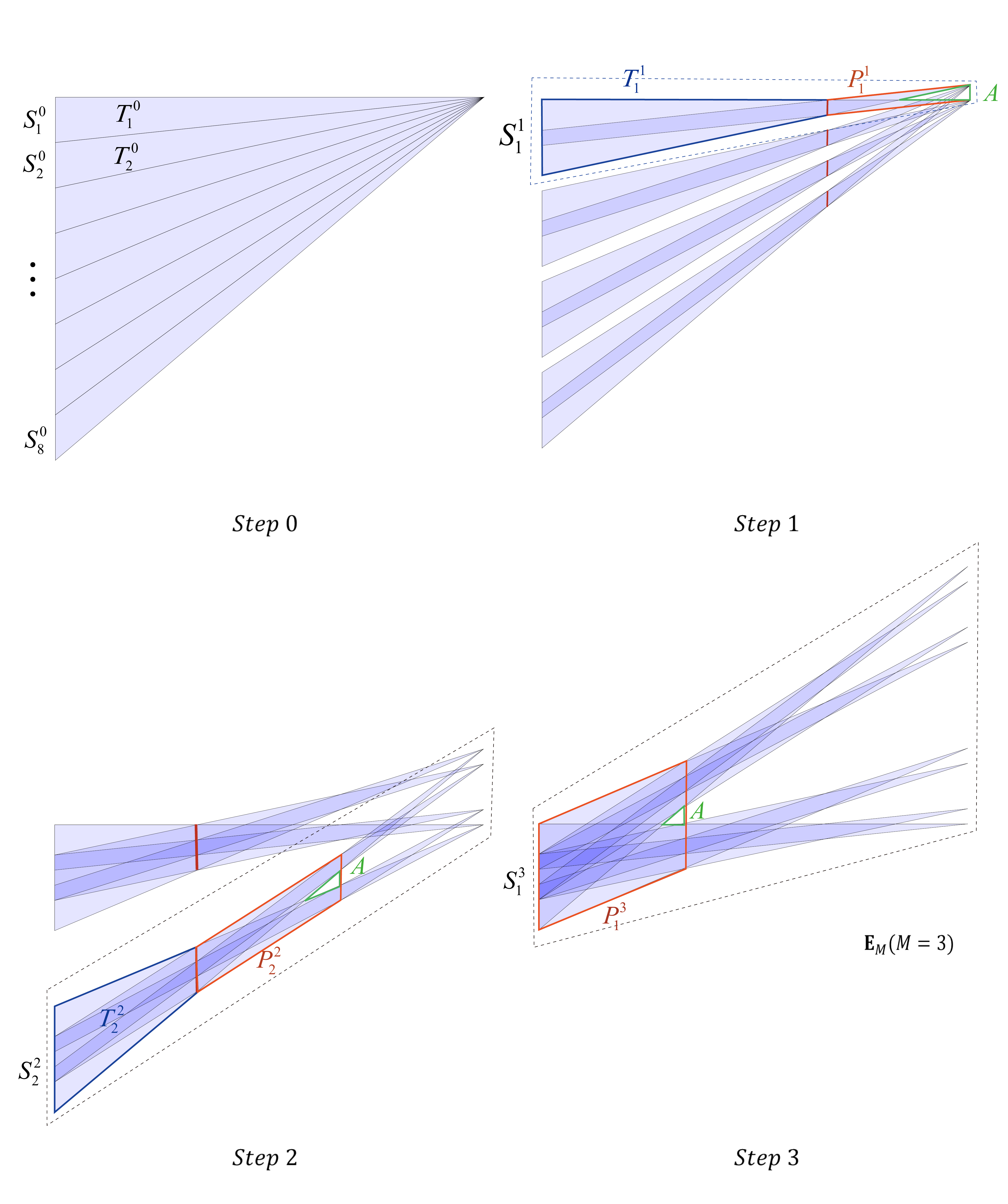}
    \caption{Construction of $\Eb_3$}
    \label{fig:E}
    \end{figure} 
  
\medskip 

\subsection{Properties of $\Eb_{\n}$} 
First, we make some basic observations about the construction. 

\begin{proposition}
\label{prop:EM-observe}
$({i})$ Each vertical slide in Step $\rk$ is of size 
$\frac{\rk}{\n}\frac{2^{\rk-1}}{2^\n}$; 

\noindent$({ii})$ The left $($larger$)$ base of $T_{i}^{(\rk)}$ has length $\Big(1-\frac{j-1}{M}\Big)\frac{2^\rk}{2^\n}-\frac{1}{\n}\frac{1}{2^\n}$; 

\noindent$({iii})$
The vertical side of 
$P_{i}^{\rk}$ has length
$\frac1\n \frac{2^\rk}{2^\n}
-\frac{1}{\n}\frac{1}{2^\n}$;

\noindent$({iv})$
The vertical side of $A_{i}^{\rk}$ has length $\frac{1}{\n}\frac{1}{2^\n}$.
\end{proposition}

\begin{proof}
    The proofs of (\textit{i})$-$(\textit{iv}) are straightforward, 
so we omit them here. 
\end{proof}

In what follows, 
we use the notation $T_{i}$ 
(resp. $T_{i}^{(\rk)}$, $P_{i}^{\rk}$, $A_{i}^{\rk}$) 
to denote the final position (as a set) of 
$T_{i}^{(0)}$ 
(resp. $T_{i}^{(\rk)}$, $P_{i}^{\rk}$, $A_{i}^{\rk}$) 
in the construction.  
Therefore, 
\eq{eq:Em=UTi}{\Eb_{\n}=\bigcup_{i=1}^{2^M} T_{i}.}
The upper side of the triangle $T_{i}$ is a line segment 
\begin{equation}
\label{eq:line-eqn}
\ell_{\alpha_i}(x):=\alpha_i x + \beta_i,\quad0\le x\le 1,
\end{equation}
where $i=1,\cdots, 2^\n,$ 
\begin{align}
\alpha_i &= \frac{i-1}{2^\n} 
=\sum_{k=1}^\n \frac{\varepsilon_k}{2^k},\quad \varepsilon_k \in \{0,1\}, \notag \\
\beta_i &= -\alpha_i
+ \sum_{k=1}^\n \varepsilon_k \,\frac{2^{\n-k}}{2^\n}\left(1-\frac{k-1}{\n}\right)
= -\frac{1}{\n}\sum_{k=1}^\n \frac{(k-1)\varepsilon_k}{2^k}. \notag 
\end{align}
Note that by Proposition \ref{prop:EM-observe}\,($ii$) (with $j=M$), we have 
\begin{equation}
\label{eq:b>-1/M}
-\frac1\n<\beta_i\le 0. 
\end{equation} 
The upper side of $P^j_i=P^j(\alpha)$ (see \eqref{eq:Pij-slope}) 
is given by $\ell_{\alpha}$ defined above. 
Correspondingly, 
the upper side of $T^{(j)}_i$ 
and the lower side of $A^j_i$ are given by $\ell_{\alpha-\epsilon_j}$, 
where 
\begin{equation}
\label{eq:epsilon_j}
\epsilon_j:=\frac{1}{2^{M-j+1}},\quad j=1,\cdots,M.
\end{equation}

\smallskip

The $\rk$-\textit{th band of} $\Eb_{\n}$ is defined by 
$$\Eb_{\n}^{(\rk)}=\Eb_{\n}\cap
\left\{(x,y): x_j\le x\le x_{j-1},\, y\in\mathbb R\right\},
\quad \rk=1,\cdots, \n.$$
Clearly, 
$$\mathbf E_M=\bigcup_{j=1}^M \mathbf E^{(j)}_M.$$
Note that by construction,  
\begin{equation}
\label{eq:E=P-A}
\Eb_{\n}^{(\rk)} 
= \bigcup_{i=1}^{2^{\n-\rk}} \big(P_{i}^{\rk}-A_{i}^{\rk}\big).
\end{equation}
In particular, we have 
\begin{equation}
\label{eq:E-in-P}
\Eb_{\n}^{(\rk)}\subset\Pb_{\n}^{(\rk)}:=
\bigcup_{i=1}^{2^{\n-\rk}} P_{i}^{\rk}, 
\end{equation}
and, consequently, 
\begin{equation}
\label{eq:Em-in-Pm}
\Eb_{\n}\subset\Pb_M:=\bigcup_{j=1}^M \Pb^{(j)}_M.
\end{equation}

\smallskip

The following proposition shows that $\Pb_{\n}$ has a Perron-tree structure, 
with $P_{i}^{\rk}$ representing the `branches' 
and $A_{i}^{\rk}$ representing the `bifurcations'. 
Note that, 
for $2\le j\le M$, each $P^j_i=P^j(\alpha)$ in $\Pb_{\n}^{(\rk)}$ is connected to two parallelograms 
$P^{j-1}(\alpha\pm\epsilon_{j-1})$ in $\Pb_{\n}^{(\rk-1)}$. 

\begin{proposition}\label{prop:Perron}
For each $\rk=1,\cdots, \n$, 
the parallelograms $P_i^{\rk}\subset \Pb_{\n}^{(\jk)}$ 
$(i=1,\cdots, 2^{\n-\jk})$ are disjoint and 
arranged monotonically \emph{(}ascendingly\emph{)} 
with respect to slope. 
\end{proposition}

\begin{proof} 
The case $j=M$ is trivial 
since there is only one parallelogram in $\Pb_{\n}^{(M)}$. 
So we assume $j\le M-1$ below. 
By \eqref{eq:Pij-slope}, the slopes of any two parallelograms in $\Pb_{\n}^{(\jk)}$ differ by a multiple of $\frac{1}{2^{\n-\jk}}$. 
Let 
$$\alpha=\displaystyle \sum_{k=1}^\n \frac{\varepsilon_k}{2^k},\quad \widetilde{\alpha}=\displaystyle \sum_{k=1}^\n \frac{\widetilde\varepsilon_k}{2^k}$$ 
be the slopes of two such parallelograms. 
Without loss of generality, assume that they are adjacent, that is, 
$$\alpha-\widetilde \alpha= \frac{1}{2^{\n-\jk}}.$$ 
Let ${k_0} \in \{1, 2, \cdots, \n-\jk \}$ be the binary place such that 
$$
\begin{cases}
\varepsilon_k=\widetilde\varepsilon_k, & k=1, 2, \cdots, {k_0}-1; \\
\varepsilon_{k_0}=1,\;\widetilde\varepsilon_{k_0}=0;&\ \\
\varepsilon_k=0,\; \widetilde\varepsilon_k=1, &k={k_0}+1, \cdots, \n-\jk. 
\end{cases}
$$
Then, by \eqref{eq:line-eqn},
\begin{align*}
\ell_\alpha\big(x_j\big)-\ell_{\widetilde \alpha}\big(x_j\big)
&=\frac{\n-\jk-{k_0}+1}{\n} \frac{1}{2^{k_0}} - \displaystyle \sum_{k={k_0}+1}^{\n-\jk} \frac{\n-\jk-k+1}{\n} \frac{1}{2^k} \\
&=\frac{2}{\n}\frac{1}{2^{k_0}}-\frac{1}{\n}\frac{1}{2^{\n-\jk}}. 
\end{align*}
Since ${k_0} \leq \n-j$, it follows that 
\begin{align*}
\ell_\alpha\big(x_j\big)-\ell_{\widetilde\alpha}\big(x_j\big) 
\ge \frac{1}{\n}\frac{1}{2^{k_0}}
>\frac1\n \frac{1}{2^{\n-j}}-\frac{1}{\n}\frac{1}{2^\n}.
\end{align*}
By Proposition \ref{prop:EM-observe}\,($iii$), this implies that the lower-left vertex of the parallelogram $P^j(\alpha)$ lies strictly above the upper-left vertex of $P^j(\widetilde\alpha)$ (cf. Figure \ref{fig:E}). 
Since the vertical distance between $P^j(\alpha)$ and $P^j(\widetilde\alpha)$ is minimized at $x=x_j$, and since the adjacent pair $(\alpha,\widetilde\alpha)$ is arbitrary, 
we thus conclude that the parallelograms in $\Pb_{\n}^{(\jk)}$ are arranged ascendingly with respect to slope. 
This completes the proof of 
Proposition \ref{prop:Perron}. 
\end{proof} 

\begin{remark}
\label{rmk:Perron}
We make two observations from the proof above. 

\noindent $({i})$ The vertical distance between adjacent parallelograms $P^j(\alpha), P^j(\widetilde\alpha)$ in $\Pb^{(j)}_M$ 
is attained at $x=x_j$, and is given by 
\begin{equation}
    \label{eq:vertical-gap}
    \frac{2}{\n}\frac{1}{2^{k_0}}-\frac{2}{\n}\frac{1}{2^{\n-\jk}}+\frac{1}{\n}\frac{1}{ 2^\n};
\end{equation}
in particular, it is at least 
$\frac{1}{\n}\frac{1}{2^{\n}}$, 
and does not exceed $\frac{1}{\n}$. 

\noindent $({ii})$ 
When $\alpha-\widetilde{\alpha}={2^{-M}}$, by a similar argument, one obtains  
\begin{equation} \label{eq:intercept-difference} \ell_{\alpha}(0) - \ell_{\widetilde{\alpha}}(0) 
= \frac{2^{\n-{k_0}+1}-\n-1}{\n \cdot 2^\n};  
\end{equation}
in particular, 
$\vert \ell_{\alpha}(0) - \ell_{\widetilde{\alpha}}(0) \vert \geq \frac{1}{M\cdot 2^M }$ always holds if $M$ is even. 
\end{remark}

For $0\le x\le 1$, denote the cross-sections (see \eqref{eq:E|x})
$$\Eb_{\n}^x=\Eb_{\n}|_x,\quad  
\Pb_{\n}^x=\Pb_{\n}|_x,$$ 
and 
\eq{eq:ExPxTx}{
T_i^x=T_i|_x\quad (i=1,\cdots,2^M).
}
When $x$ is fixed, 
we shall often view these cross-sections as subsets of $\mathbb R$. 
Clearly, 
$$\Eb_{\n}^x=\bigcup_{i=1}^{2^M} T_{i}^x$$ 
is a union of $2^M$ many intervals of length $\frac{1-x}{2^M}$. 
When $x_j<x\le x_{j-1}$, 
by \eqref{eq:E-in-P} and Proposition \ref{prop:EM-observe}\,($iii$), $\Eb_{\n}^x\subset\Pb_{\n}^x$ can be covered by $2^{M-j}$ many intervals 
of length $\frac{2^\rk}{\n \cdot 2^\n}-\frac{1}{{\n}\cdot 2^\n}$. Therefore  
\begin{align}
\label{eq:1/N}
|\Eb_{\n}^x|\le |\Pb_{\n}^x|< 2^{\n-\jk} \, \frac{2^{\jk}}{\n \cdot 2^{\n}}
=\frac{1}{\n}. 
\end{align} 
Note that \eqref{eq:1/N} holds uniformly for $j=1, \cdots, M$.  
Integrating \eqref{eq:1/N} over $x$ yields 
$\vert \Eb_{\n} \vert <\frac{1}{\n}$.  As shown in \cite{Keich1999}, 
this is sufficient for bounding the generalized Minkowski dimension of the corresponding Besicovitch sets (see Proposition \ref{proposition of Fn}\,($iii$) for details). 

\begin{remark}
\label{rmk:eta-connected}
Let $P^j(\alpha)$ be a parallelogram in $\Pb_{\n}^{(\rk)}$. 
For $0\le x\le x_j$, consider the intervals
\begin{equation}
\label{eq:Tx-intervals}
\big\{T_i^x: T_i\cap P^j(\alpha)\neq\varnothing\big\}.
\end{equation}
By the construction of $\Eb_M$ (cf. Figure \ref{fig:E}), 
the union of these intervals is the cross-section of the trapezoid $T^{(j)}_i$, 
thus is connected. 
Moreover, by Corollary \ref{cor:eta-connected} and induction on $j$, 
the interval family \eqref{eq:Tx-intervals} is $\frac{1}{M\cdot 2^M}$-connected in the sense of Definition \ref{def:eta-connected}. 
\end{remark} 

In order to study the generalized Hausdorff dimension, 
we shall need further distributional properties of $\Eb_{\n}^x$.  
In what follows, denote 
\begin{equation}
\label{eq:rj-eqn}
r_j=\frac{2^j}{\n\cdot 2^\n}.
\end{equation} 
Note that $r_0=\frac{1}{{\n}\cdot 2^\n}$ 
and that $r_j = 2^j r_0$ grows dyadically with $j$. 
The relevance of the scales 
$r_j\, (j=1,\cdots, M)$ 
is that they approximately give 
the vertical widths of the parallelograms 
in $\Pb^{(j)}_M$ (only off by $r_0$ according to Proposition \ref{prop:EM-observe}\,($iii$)). 

The following proposition shows that each cross-section
$\Eb_{\n}^x\subset\Eb_M^{(j)}$ forms a Cantor-like set 
at scale $r_j$ and above (up to $x_{j-1}$). 
For convenience, we shall work with $\Pb_{\n}^x$ in place of $\Eb_{\n}^x$. 
Let $I\subset\mathbb R$ be an open interval, 
denote by $N_x(I)$ the number of connected components of $\Pb_{\n}^x\cap I$. 
A connected component of the complement
$\mathbb R\backslash \Pb_{\n}^x$ will be called a 
\textit{gap of} $\Pb_{\n}^x$ 
and denoted by $w$ (which could be infinite).

\begin{proposition}
\label{prop NI}
Suppose $x_j< x\le x_{j-1}$ $(j=1,\cdots, M-1)$.
Let $I$ be an open interval with 
$r_j\le|I|\le x_{j-1}$, 
and let $w_{max}$ be the largest gap(s) of $\Pb_{\n}^x$ that intersects $I$. 
If $ \vert I \cap w_{max} \vert\leq \frac{\vert I \vert}{2}$, then 
\begin{equation}
\label{eq:N(I)-bounds}
\ConstOne\,
\dfrac{\frac{\vert I \vert}{r_j}}{\log \frac{\vert I \vert}{r_j}} \leq N_x(I) 
\leq 
\ConstTwo\,
\dfrac{\frac{\vert I \vert}{r_j}}{\log \frac{\vert I \vert}{r_j}}. 
\end{equation}
\end{proposition}

\begin{proof}
By Proposition \ref{prop:Perron}, 
the gaps of $\Pb^x_M$ increase as $x$ increases from $x_j$ to $x_{j-1}$. 
Therefore, it suffices to prove the proposition for 
$x=x_j^+$ 
(the general case follows by monotonicity considerations). 

To be more precise, when $x=x_j^+$, $N_{x_j^+}(I)$ is understood as the number of connected components of $\Eb_{\n}^{x_j}\cap I$. 
For simplicity, we denote $C_j=\Eb_{\n}^{x_j}$. 
Note that 
$C_j$ is a union of $2^{\n-\jk}$ many disjoint intervals of length $r_j-r_0$. 
By Remark \ref{rmk:Perron}\,($i$), 
the gaps between the adjacent intervals in $C_j$ have size given by \eqref{eq:vertical-gap}. 
It follows that (the sizes of) these gaps, ordered vertically, 
form an automatic sequence given by the final outcome of the iteration: 
$$
\begin{cases}\label{sequence}
W_1^{(\jk)}=\big\{r_0\big\},\\[1em]
W_{2^{\ell}-1}^{(\jk)}
=\left\{W_{2^{\ell-1}-1}^{(\jk)},\, 
2^\ell r_j- 2r_j + r_0,\,
W_{2^{\ell-1}-1}^{(\jk)} \right\},\quad 
\ell=2,\cdots, \n-\jk. 
\end{cases}
$$
Note that the subscript is chosen so that 
$\#W_{2^{\ell}-1}^{(\jk)}=2^{\ell}-1$. 
Note also that $$\max W_{2^{\ell}-1}^{(\jk)}= 2^{\ell} r_j- 2r_j + r_0$$ 
grows with $\ell$, 
and is comparable to 
$2^{\ell-1} r_j$ 
when $\ell\ge 2$ (cf. Figure \ref{fig:loopback}).
\begin{figure}[H]
    \centering
\includegraphics[width=0.8\linewidth]{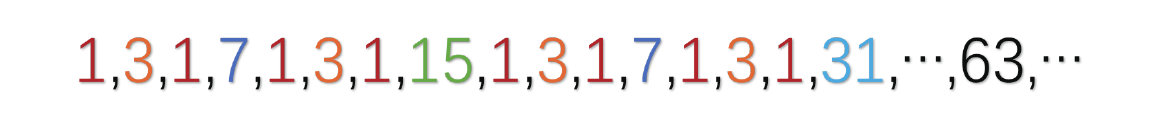}
    \caption{The gaps in $C_0$, scaled by $M\cdot 2^M$ ($M\ge 5$)}    \label{fig:loopback}
\end{figure}
Now consider the open interval $I_\ell^{(\jk)}\, (\ell=1,2,\cdots, \n-\jk+1)$ 
which has the same lowest endpoint with $C_j$ and length $\ell\cdot 2^{\ell-1} r_j$ (see Figure \ref{fig:interval}). 
Using the automatic sequence above, it is easy to verify that 
$$N_{x_j^+}\big(I_\ell^{(\jk)}\big)
=2^{\ell-1}.$$ 
Therefore, 
$$
\frac{\frac{|I_\ell^{(\jk)}|}{r_j}}{\log\frac{|I_\ell^{(\jk)}|}{r_j}}
\le N_{x_j^+}\big(I_\ell^{(\jk)}\big)
\le 2\frac{\frac{|I_\ell^{(\jk)}|}{r_j}}{\log\frac{|I_\ell^{(\jk)}|}{r_j}}.$$
This shows that 
\eqref{eq:N(I)-bounds} holds when $I=I_\ell^{(\jk)}$. 
The general case can be reduced to this case using standard doubling and self-similarity arguments; 
we omit the details.  
\end{proof}

\begin{remark}
\label{rmk:NI-condition}
\noindent $({i})$ The proof above shows that the upper bound in \eqref{eq:N(I)-bounds} holds as long as $|I|\ge r_j$; 

\noindent $({ii})$ The condition $\vert I \cap w_{max} \vert\leq \frac{\vert I \vert}{2}$ is satisfied in particular when $c(I)\in \Pb_\n^x$, 
where $c(I)$ denotes the center of $I$. 
\end{remark}

  \begin{figure}[ht]
  \centering
    \begin{minipage}[b]{0.5\textwidth}
    \centering
    \includegraphics[width=\textwidth]{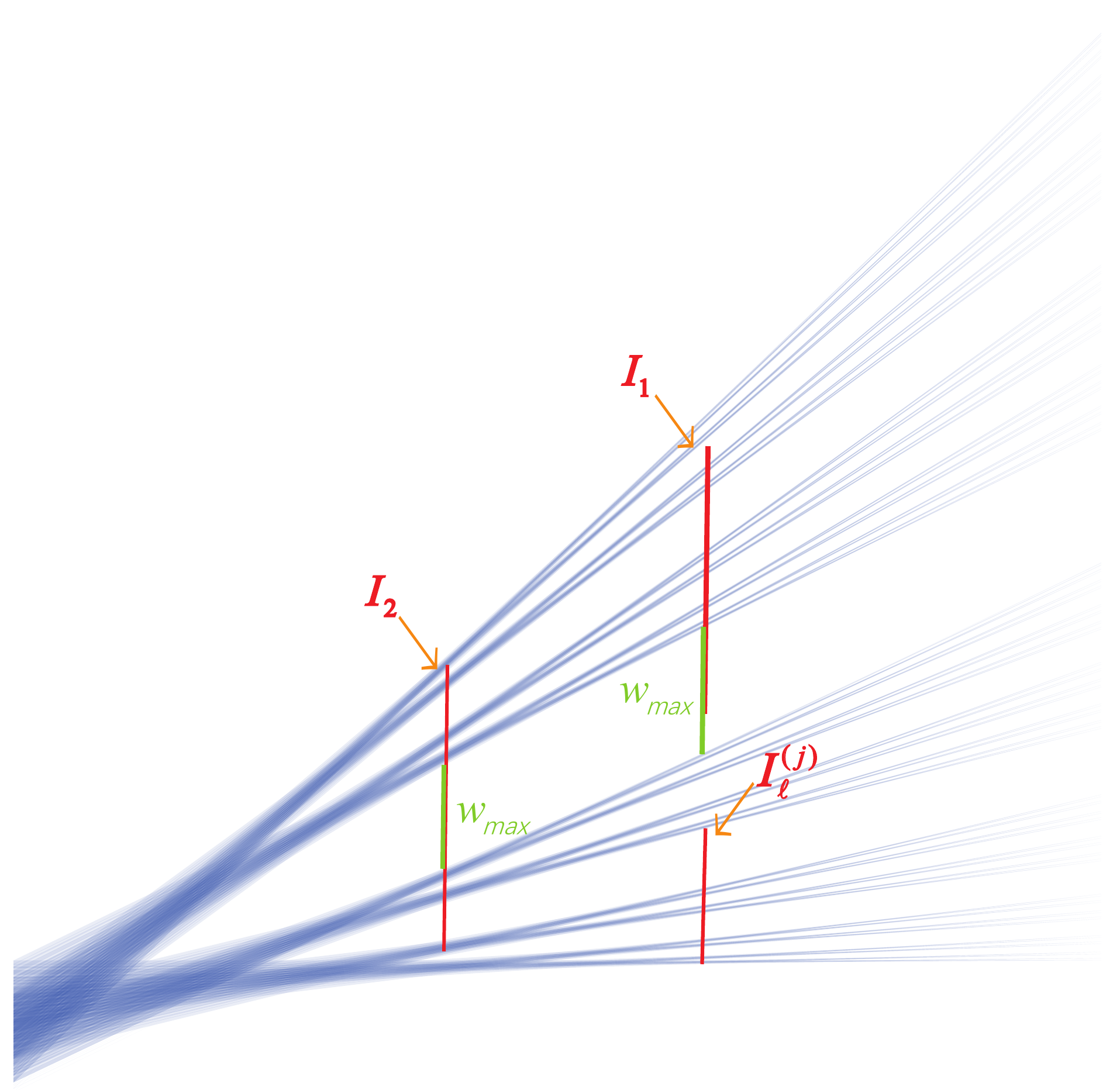}
    \caption{Possible positions of $I$} 
    \label{fig:interval}
    \end{minipage} 
    \begin{minipage}[b]{0.48\textwidth}
    \centering
    \includegraphics[width=\textwidth]{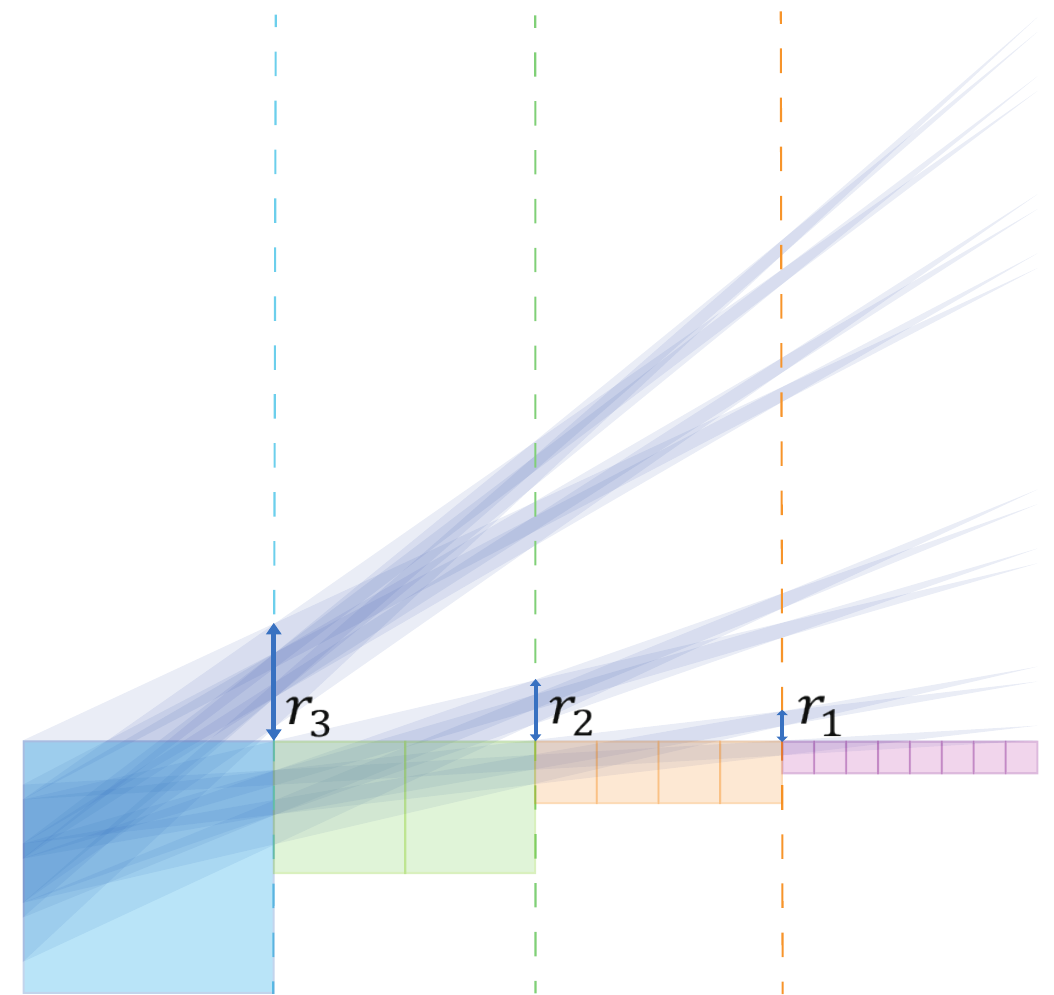}
    \caption{Adaptive cover of $\Eb_4$}
    \label{fig:adaptive}
    \end{minipage}
  \end{figure}  

We now turn to distributional properties of $\Eb_{\n}$. 
In what follows, 
$\n$ will always take the form
$\n=2^m$
for some integer $m\ge 1$. 
As in Section \ref{sec:prelim}, 
we shall be concerned with \textit{dyadic covers} of $\Eb_{\n}$, i.e. 
covers of the form $\{Q_i\}$, 
where the dyadic squares $Q_i$ are mutually disjoint 
and satisfy $Q_i\cap \Eb_{\n}\neq\varnothing$. 

Recall from \eqref{eq:rj-eqn} that 
$r_j=2^{-(M-j)-m}$, 
which is now a dyadic scale. 
Similarly, $x_j=1-\frac{j}{2^m}$ is now a dyadic rational. 

\begin{definition}[adaptive cover]
For $j=1,\cdots, \twoM$, 
cover $\Eb_{\twoM}^{(j)}$ by 
dyadic squares $\{Q_i^{(j)}\}$ 
with 
$$Q_i^{(j)}\subset
[x_j, x_{j-1}]\times\mathbb R,
\quad r(Q_i^{(j)})=r_{j},\quad \interior(Q_i^{(j)})\cap \Eb_{\twoM}^{(j)}\neq\varnothing.$$ 
Set $\{Q_i^{adp}\} = \bigcup_{j=1}^{\twoM} \{Q_i^{(j)}\}$. 
We call $\{Q_i^{adp}\}$ 
the \emph{adaptive cover of} $\Eb_{\twoM}$. 
\end{definition} 

More generally, 
a dyadic square $Q$ is called an \textit{adaptive square} if 
$$Q\subset
[x_j, x_{j-1}]\times\mathbb R\quad\text{and}\quad
r(Q)=r_{j}$$ 
for some $j\in\{1,\cdots,M\}$. 
A union of dyadic squares $F=\bigcup_i Q_i$ 
is called \textit{sub-adaptive} if each $Q_i$ either satisfies $r(Q_i)>\frac1M$ or satisfies 
$$Q\subset
[x_j, x_{j-1}]\times\mathbb R\quad\text{and}\quad  r(Q)\ge r_{j}$$
for some $j\in\{1,\cdots,M\}$.  

The adaptive cover 
$\{Q_i^{adp}\}$ is a quasi-optimal cover of $\Eb_{\twoM}$ 
with respect to the gauge 
$$h(r):=r^2 \log (1/r),$$ 
in the sense of Definition \ref{def:optimal-cover}. 
To show this, we first give a definition. 

\begin{definition}[adaptive measure] 
Given $\{Q_i^{adp}\}$, the 
\emph{adaptive measure of} $F\subset\mathbb R^2$ 
is defined by 
\begin{equation}
\label{eq:adaptive-measure}
\adp(F)=\sum_{Q\in\{Q_i^{adp}\}:\atop \interior(Q) \cap F \neq \varnothing} h(Q).
\end{equation}
\end{definition}
Note that for any adaptive square $Q$, 
$$\adp(Q)=\begin{cases}
h(Q),    &\text{if }Q\in\{Q_i^{adp}\}; \\
0,    &\text{otherwise}.
\end{cases}$$
More generally, for any sub-adaptive set 
$F$, we have 
\eq{eq:adp(F)}{
\adp(F)=\sum_{Q\in\{Q_i^{adp}\}:\, Q\subset F} h(Q). 
}
As a consequence, 
$\adp(\cdot)$ is finitely additive over sub-adaptive sets.   

Below 
we shall only be concerned with 
$\Eb_{\twoM}\cap\{\frac12\le x\le \frac34\}$. 
This is without loss of generality since one can use dilation to obtain line segments of length (at least) one and with slope ranging from 0 to 1. 
For simplicity, 
we still use $\Eb_{\twoM}$ to denote this part of $\Eb_{\twoM}$, 
in other words, 
\begin{equation}
\label{eq:13-34}
\Eb_{\twoM}=\Eb_{\twoM}\cap\Big(\Big[\frac12,\frac34\Big]\times\mathbb R\Big).
\end{equation} 
Similarly, we write 
$$\{Q_i^{adp}\}=\left\{Q\in\{Q_i^{adp}\}: 
Q\subset
\Big[\frac12, \frac34\Big]\times\mathbb R
\right\}.$$ 

Let $Q$ be a dyadic square, write 
\eq{eq:(2k+1)Q}{
(2k+1)Q=\bigcup_{\varepsilon\in\{0,\pm1,\cdots,\pm k\}} Q+(0, \varepsilon r(Q)),\quad k\ge 1.
}

\begin{proposition}
\label{prop:optimum cover of EM}
Let $Q^{(j)}\in \{Q_i^{adp}\}$. \\
\noindent$(i)$ \emph{(small squares)}. Let $\{U_k\}$ be a dyadic cover of 
$(5Q^{(j)})\cap \Eb_{\twoM}$. 
Then 
$$\adp (Q^{(j)})\leq 2\sum_{k} h(U_k).$$
\noindent$(ii)$ Let $U$ be a dyadic square with 
$Q^{(j)}\subset
U.$\\  
\indent\; $(a)$ \emph{(intermediate square)}. If  $r(U) \le \sqrt{r_j}$, 
then 
$$\adp (U) \leq 96\, h(U);$$ 
\indent\; $(b)$ \emph{(large square)}. If $\sqrt{r_j} \le r(U) \leq \frac14$, then we have 
$$
\adp (U) \leq 96\,\vert U \vert
\quad\text{and}\quad 
|U|\le 24\,\adp ({7}U).$$
\end{proposition}

\begin{proof} 
\begin{figure}[htb]
  \centering
  \includegraphics[width=0.5\linewidth]{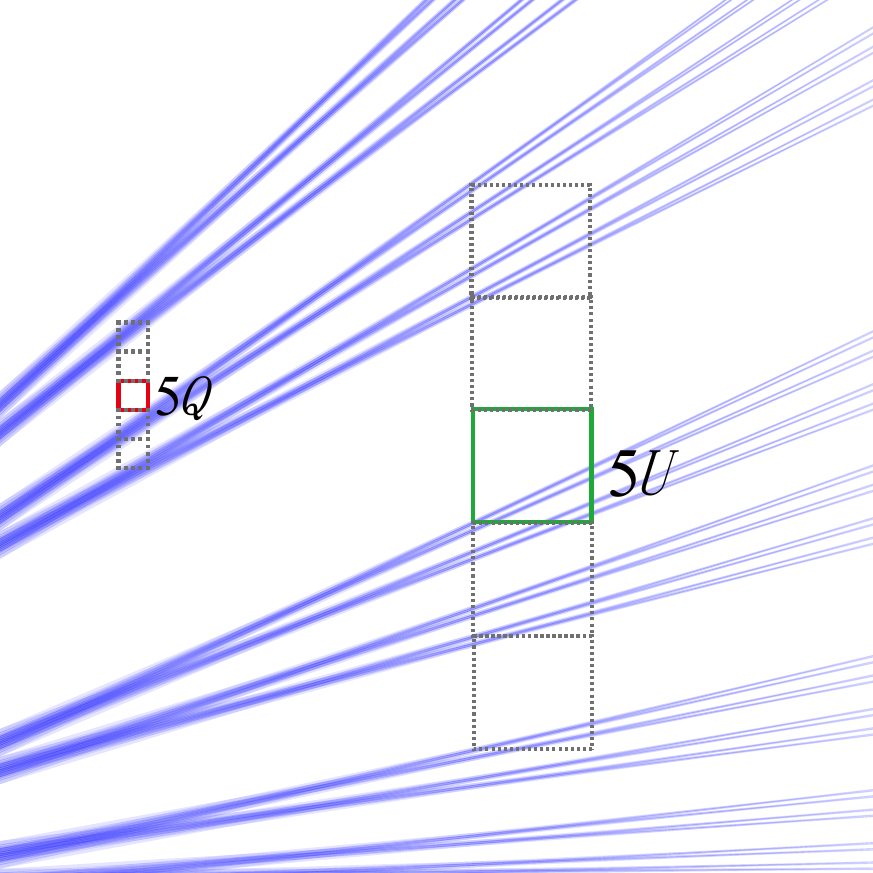}
  \caption{The squares $Q$ and $U$ are extended vertically in order to avoid `tangential' intersection}
  \label{fig:5Q-3U}
\end{figure}  
\noindent($i$) 
Since $\adp(Q^{(j)})=h(Q^{(j)})=h(r_j),$
it suffices to consider the case where $r(U_k)<r_j,\,\forall k$. 
Fix a parallelogram $P^j_i$ in $\Pb_{\twoM}^{(j)}$ that intersects $Q^{(j)}$. Since the slope of $P^j_i$ is between 0 and 1, there must be a segment of $P^j_i$ of horizontal width $r_j$ that is covered by $5Q^{(j)}$. 
Therefore,  
\begin{align*}
|(5Q^{(j)})\cap \Eb_{\twoM}|
&\ge |(5Q^{(j)})\cap P^j_i| - |(5Q^{(j)})\cap A^j_i|\\
&\ge r_j(r_j-r_0)-r_jr_0\\
&\ge \frac12 r_j^2\\
&=\frac12 |Q^{(j)}|. 
\end{align*}
Decompose $5Q^{(j)}$ into 5 squares $\widetilde Q$, 
and apply Lemma \ref{lem:leb-solidity} to each square. 
We obtain 
$$\sum_{k} h(U_k)\ge \frac12 h(Q^{(j)}).$$
This proves ($i$). 

\noindent($ii$) The case $M=2$ can be verified directly. So we assume $M\ge 4$ below. 
First, we consider the case $r(U) \leq \frac{1}{\twoM}$ (note that $\sqrt{r_j}\le \frac1M$). In this case, we have 
$$
Q^{(j)}\subset
U\subset [x_j,x_{j-1}]\times\mathbb R.
$$
Denote 
$$N(U)=\#\big\{P^j\subset\Pb_M^{(j)}: P^j\cap \interior(U)\neq \varnothing\big\}.$$ 
Notice that every adaptive square $Q_i^{(j)}\subset U$ must intersect a parallelogram $P^j\subset\Pb_M^{(j)}$ with $P^j\cap \interior(U)\neq\varnothing$. 
Notice also that no more than 3 adaptive squares $Q_i^{(j)}$ can intersect the same $P^j$ while being aligned vertically. 
Therefore, 
\eq{eq:(ii)-1}{\adp(U)\le 
3\,N(U)\, \frac{r(U)}{r_j}\, h(r_j).
} 
Let $x_{_U}$ be the $x$-coordinate of the right side of $U$, and let $I$ be the interior of the interval 
$$U|_{x_{_U}}\cup \big(U|_{x_{_U}}+r(U)\big).$$ 
Since the slope of each $P^j$ is between 0 and 1, we have 
\eq{eq:(ii)-2}{
N(U)\le N_{x_{_U}}(I).
}
On the other hand, by Proposition \ref{prop NI} (see Remark \ref{rmk:NI-condition}\,($i$)), 
\eq{eq:(ii)-3}{
N_{x_{_U}}(I)\le \ConstTwo\, 
\dfrac{\frac{2r(U)}{r_j}}{\log \frac{2r(U)}{r_j}}.
}
Combining \eqref{eq:(ii)-1}, \eqref{eq:(ii)-2}, and \eqref{eq:(ii)-3}, 
we obtain 
\eq{eq:(ii)-4}{
\adp(U)\le 
48\,r(U)^2\,\dfrac{\log\frac{1}{r_j}}{\log \frac{2r(U)}{r_j}}.
}
In particular, 
when $r(U)\le\sqrt{r_j}$, 
\eqref{eq:(ii)-4} gives 
$$\adp(U)\le 
96\,r(U)^2\log\frac{1}{r(U)}
=96\,h(U)$$ 
(using $\log R \ge 1$); 
when $r(U)\ge\sqrt{r_j}$, 
\eqref{eq:(ii)-4} gives
\eq{eq:(ii)-|U|}{
\adp(U)\le 
96\,r(U)^2
=96\,|U|
}
(using $\log \frac{2r(U)}{r_j}\ge \log\frac{1}{\sqrt{r_j}}$). 

We turn to a lower bound for $\adp(5U)$. Fix $(x_{_Q},y_{_Q})\in \interior(Q^{(j)})\cap \Eb_M^{(j)}.$ 
Let $I$ be the open interval of length $r(U)$ centered at $y_{_Q}$. Notice that every parallelogram $P^j$ intersecting $\{x_{_Q}\}\times I$ 
must be fully contained in $5U$ (i.e., $P^j$ does not intersect the top and bottom of $5U$). 
Therefore, 
$$\adp(5U)
\ge {N_{x_{Q}}(I)}\,\frac{r(U)}{r_j}\,
h(r_j).$$ 
On the other hand, by Proposition \ref{prop NI} (see Remark \ref{rmk:NI-condition}\,($ii$)), we have 
$$N_{x_{Q}}(I)\ge \ConstOne\,
\dfrac{\frac{r(U)}{r_j}}{\log \frac{r(U)}{r_j}}.$$ 
Combining the last two bounds, we obtain 
\eq{eq:mu(5U)}{
\adp(5U)\ge \frac{1}{24}\,r(U)^2\, 
\dfrac{\log\frac{1}{r_j}}{\log \frac{r(U)}{r_j}}
\ge \frac{1}{24}\,|U|, 
}
where we have used $r(U)\le 1$. 
This proves ($ii$) in the case $r(U) \leq \frac{1}{\twoM}$. 

Now we consider the case  $r(U)>\frac{1}{M}$. In this case only part ($b$) needs to be proved, 
since the case $r(U)\le\sqrt{r_j}\; (\le \frac1M)$ is ready covered above. 
For the first bound, we decompose 
$$
U=\bigcup_k\, U_k
$$
where each $U_k$ is a dyadic square with sidelength $\frac1M$. 
Note that, by definition, 
\eq{eq:(ii)-6}{
\adp(U)=\sum_k \adp(U_k). 
}
Apply \eqref{eq:(ii)-|U|} to each $U_k$ with $\adp(U_k)\neq 0$. 
We have  
$$\adp(U_k)\le 96\,|U_k|.$$ 
Thus, combining this with \eqref{eq:(ii)-6}, we obtain
$$\adp(U) 
\le 96 \sum_k |U_k|
= 96\,|U|,$$  
as desired. 

The proof for the second bound is similar to the case $r(U)\le\frac1M$ above. 
Fix $(x_{_Q},y_{_Q})$ as before, and fix a line segment $l\subset\Eb_M$ that passes through $(x_{_Q},y_{_Q})$. 
Decompose 
\begin{align*}
7U
&=\bigcup_{j}\,[x_j,x_{j-1}]\times (7U)|_{x_{_Q}}\\
&=:\bigcup_{j} R_j.
\end{align*}
For each $j$, fix a point 
$(x_{_{Q_j}},y_{_{Q_j}})\in l\cap \interior(R_j)$. 
Let $I_j$ be the open interval of length $r(U)$ centered at $y_{_{Q_j}}$. Then every parallelogram $P^j\subset\Pb_M^{(j)}$ that intersects $\{x_{_{Q_j}}\}\times I_j$ 
must be fully contained in $R_j$. Therefore, by Proposition \ref{prop NI}, we have 
\begin{align*}
\adp(R_j)
&\ge 
{N_{x_{Q_{j}}}(I_j)}\,\frac{1/M}{r_j}\,
h(r_j)\\
&\ge \frac{1}{24}\,
\dfrac{\frac{r(U)}{r_j}}{\log \frac{r(U)}{r_j}}\,\frac{1/M}{r_j}\,h(r_j)\\
&=\frac{1}{24}\,\frac{r(U)}{M}\,
\dfrac{\log\frac{1}{r_j}}{\log \frac{r(U)}{r_j}}\\
&\ge \frac{1}{24}\,\frac{r(U)}{M}.
\end{align*}
Summing over $j$, we obtain 
\begin{align*}
\adp(7U)
&=\sum_j\adp(R_j)\\
&\ge \frac{1}{24}\,\frac{r(U)}{M}\,\frac{r(U)}{1/M}\\
&=\frac{1}{24}|U|.
\end{align*} 
This proves ($ii$), and the proof of Proposition \ref{prop:optimum cover of EM} is complete. 
\end{proof}

\begin{remark} 
\label{remark:prop-3.9}
(i) Since $|U|\le h(U)$, it follows from Proposition \ref{prop:optimum cover of EM}\,(ii) that $$\adp(U)\le 96\,h(U)$$
holds for all dyadic squares $U\subset [x_j,x_{j-1}]\times\mathbb R$, with $r(U)\ge r_j$.\\ 
\noindent (ii) It follows from part (b) of Proposition \ref{prop:optimum cover of EM}\,(ii) that 
$$\adp(U)\le 96\,|U|$$
holds for all dyadic squares $U$ with $r(U)\ge \frac1M$.\\ 
\noindent
(iii) Applying the last bound in Proposition \ref{prop:optimum cover of EM} to 
$U=[\frac12,\frac34]\times [0,\frac14]$, one obtains $\adp(\Eb_M)\ge \frac{1}{384}$. 
Indeed, a more careful computation yields 
$$\frac{5}{32}<\adp(\Eb_\n) \leq \frac{3}{4}.$$ 
\end{remark}

Using Proposition \ref{prop:optimum cover of EM}, 
we can now apply the mass distribution principle to bound the $h$-measure of $\Eb_M$ from below. 
For simplicity, denote 
\eq{eq:c0}{
c_0=\inf_{M\ge 2}\,\adp(\Eb_\n).
}

\begin{theorem}
\label{thm:quasi-optimal}
For any $\delta\in(0,1)$, we have 
$$\mathcal M_{\delta}(\Eb_{\twoM};h)\ge \frac{c_0}{960}.$$ 
\end{theorem}

\begin{proof}
The case $M=2$ follows easily from $|\Eb_2|=\frac{1}{16}$ and 
$$\mathcal M_{\delta}(\Eb_{\twoM};h)\ge |\Eb_{\twoM}|.$$
So we assume $M\ge 4$ below. 
Recall from Section \ref{sec:prelim} that 
\eq{eq:M(Em)-0}{
\mathcal M_\delta(\Eb_M;h)=\sum_i h(Q_i^{opt}),
}
where $\{Q_i^{opt}\}$ is any (fixed) optimal $\delta$-cover of $\Eb_M$ with respect to $h$. 
Notice that, due to the optimality of $\{Q_i^{opt}\}$, 
each square $Q_i^{opt}$ satisfies 
$$\interior(Q_i^{opt})\cap \Eb_M\neq\varnothing$$
(otherwise, using the solidity of $\Eb_M$, $Q_i^{opt}$ can be removed from the cover). 
Therefore, we can classify $Q_i^{opt}$ into two cases: 
$$\begin{cases}
\text{Case 1: }\,r(Q_i^{opt})>\frac1M, \text{ or }&\ \\
\quad\quad\quad\;\,\,Q_i^{opt}\subset [x_j, x_{j-1}]\times\mathbb R, \quad r(Q_i^{opt})\ge r_j,&\text{where } \frac M4<j\le\frac M2;\\
\text{Case 2: }\,
Q_i^{opt}\subset [x_j, x_{j-1}]\times\mathbb R, \quad r(Q_i^{opt})< r_j,&\text{where } \frac M4<j\le\frac M2.
\end{cases}$$
To compare with the adaptive cover $\{Q^{adp}_i\}$, we classify $Q^{adp}\in\{Q^{adp}_i\}$ into two cases: 
$$\begin{cases}
\text{Case A: }\,
Q^{adp}\subset Q_i^{opt}, \text{ for some } Q_i^{opt}
\in\text{Case\,1;}\\ 
\text{Case B: }\,
Q^{adp}\supset Q_i^{opt}, \text{ for some } Q_i^{opt}\in\text{Case\,2.}
\end{cases}$$
Applying Proposition \ref{prop:optimum cover of EM}\,($ii$) (see Remark \ref{remark:prop-3.9}\,($i$),\,($ii$)) to each $Q_i^{opt}$ in $\text{Case\,1}$, we have 
\eq{eq:M(Em)-1}{
\sum_{Q_i^{opt}\in\text{\,Case\,1}} h(Q_i^{opt})
\ge \frac{1}{96} \sum_{Q_i^{opt}\in\text{\,Case\,1}} \adp(Q_i^{opt}). 
}
Since each $Q_i^{opt}$ in $\text{Case\,1}$ is sub-adaptive, by \eqref{eq:adp(F)} we have 
\begin{align}
\sum_{Q_i^{opt}\in\text{\,Case\,1}} \adp(Q_i^{opt}) 
&=\sum_{Q_i^{opt}\in\text{\,Case\,1}} \sum_{Q^{adp}\subset Q_i^{opt}} h(Q^{adp})\notag\\
&=\sum_{Q^{adp}\in\text{\,Case\,A}} h(Q^{adp}).
\label{eq:M(Em)-2}
\end{align}
On the other hand, 
for $Q_i^{opt}$ in $\text{Case\,2}$,  
we can write
\eq{eq:M(Em)-3}{
\sum_{Q_i^{opt}\in\text{\,Case\,2}} h(Q_i^{opt})
=\sum_{Q^{adp}\in\text{\,Case\,B}} \sum_{Q_i^{opt}\subset Q^{adp}} h(Q_i^{opt}). 
}
Combining \eqref{eq:M(Em)-1}, \eqref{eq:M(Em)-2}, and \eqref{eq:M(Em)-3}, we see that 
\begin{align}
\sum_i h(Q_i^{opt}) 
&=\sum_{\text{\,Case\,1}} h(Q_i^{opt})+\sum_{\text{\,Case\,2}} h(Q_i^{opt})\notag\\
&\ge \frac{1}{96}\sum_{\text{\,Case\,A}} \tilde\adp(Q^{adp})+\sum_{\text{\,Case\,B}} \tilde\adp(Q^{adp})\notag\\
&\ge \frac{1}{96}\sum_{Q^{adp}} \tilde\adp(Q^{adp}), \label{eq:M(Em)-4}
\end{align}
where 
\eq{eq:tilde-mu-def}{
\tilde\adp(Q^{adp})
=\begin{cases}
h(Q^{adp}), &\textit{if }\, Q^{adp}\in\text{Case\,A;}\\ \ \vspace{-0.5em}\\
\displaystyle\sum_{Q_i^{opt}\subset Q^{adp}} h(Q_i^{opt}), &\textit{if }\, Q^{adp}\in\text{Case\,B.}
\end{cases}
}
Now, writing 
\eq{eq:tilde-mu(F)-def}{
\tilde\adp(F)=\sum_{Q^{adp}\,:\atop \interior(Q^{adp})\cap F\neq\varnothing} \tilde\adp(Q^{adp}),
}
by Proposition \ref{prop:optimum cover of EM}\,($i$), we have 
\begin{align}
\sum_{Q^{adp}} \tilde\adp(Q^{adp})
&\ge\frac{1}{5} \sum_{Q^{adp}} \tilde\adp(5Q^{adp}) \notag \\
&\ge\frac{1}{10} \sum_{Q^{adp}} \adp(Q^{adp}) \notag \\
&=\frac{1}{10}\,\adp(\Eb_M). \label{eq:M(Em)-5} 
\end{align}
Combining \eqref{eq:M(Em)-0}, \eqref{eq:M(Em)-4}, and \eqref{eq:M(Em)-5}, we obtain 
$$\mathcal M_\delta(\Eb_M)
\ge \frac{1}{960}\,\adp(\Eb_M) 
\ge \frac{c_0}{960}.$$
This completes the proof of Theorem \ref{thm:quasi-optimal}.
\end{proof}

The proof of Theorem A relies on the following (local) version of Theorem \ref{thm:quasi-optimal}, which holds for `moderately' dilated copies of $\Eb_M$. 
For notational convenience, we shall denote the height of an axis-parallel rectangle $R$ by $r(R)$. 

\begin{proposition}
\label{prop:lemma of Sj-3}
Suppose 
$$\frac{1}{M}\le\delta\le\frac18.$$
Let $l(x)=ax+b$ be a line with slope $a\in[0,1)$. 
Let $\widetilde U\subset[x_j,x_{j-1}]\times [-1,1]$ be a dyadic square with 
$r(\widetilde U)=\frac1M$, whose center $c(\widetilde U)\in \convex(\delta.\Eb_M+l)$.
Then, for any $\delta_1\in(0,1)$, we have 
$$\mathcal M_{\delta_1}\big(\interior(\widetilde U)\cap(\delta.\Eb_{\twoM}+l);h\big)\ge c_1|\widetilde U|,$$ 
where $c_1=\frac{1}{960}\cdot\frac{1}{24^2}$. 
\end{proposition}

\begin{proof}
Applying Lemma \ref{lem:affine} and Lemma \ref{lem:dilation}, we have 
\begin{align}
&\mathcal M_{\delta_1}\big(\interior(\widetilde U)\cap(\delta.\Eb_{\twoM}+l);h\big) \notag\\
\ge &\frac13\cdot\mathcal M_{\delta_1}\big(\interior(\widetilde U-l)\cap \delta.\Eb_{\twoM};h\big)\notag\\
\ge &\frac13\,\delta\cdot \mathcal M_{\delta_1}\big(\interior(\delta^{-1}.(\widetilde U-l))\cap \Eb_{\twoM};h\big).\label{eq:mu(Sj)-1}
\end{align}
Let $\widetilde S$ be the square concentric with $\widetilde U-l$, of sidelength $r(\widetilde S)=\frac12 r(\widetilde U)$ 
(note that $\widetilde S$ is not necessarily dyadic). 
Then 
$\widetilde S\subset \interior(\widetilde U-l)$. Therefore,  
\eq{eq:mu(Sj)-1.5}{
\eqref{eq:mu(Sj)-1}\ge \frac13\,\delta\cdot \mathcal M_{\delta_1}\big((\delta^{-1}.\widetilde S)\cap \Eb_{\twoM};h\big). 
}
As in the proof of Theorem \ref{thm:quasi-optimal} above, 
write 
\eq{eq:mu(Sj)-2}{
\mathcal M_{\delta_1}\big((\delta^{-1}.\widetilde S)\cap \Eb_{\twoM};h\big)=\sum_i h(Q_i^{opt}),
}
where $\{Q_i^{opt}\}$ is an optimal $\delta_1$-cover of $(\delta^{-1}.\widetilde S)\cap \Eb_M$ with respect to $h$. 
Now classify $Q_i^{opt}$ into two cases: 
$$\begin{cases}
\text{Case 1: }\, 
r(Q_i^{opt})\ge r_j;\\
\text{Case 2: }\, 
r(Q_i^{opt})< r_j, 
\end{cases}$$
and, correspondingly, classify 
$Q^{adp}\in\{Q^{adp}_i:Q^{adp}_i\subset \delta^{-1}.\widetilde S\}$ into
$$\begin{cases}
\text{Case A: }\,
Q^{adp}\subset Q_i^{opt}, \text{ for some } Q_i^{opt}
\in\text{Case\,1;}\\ 
\text{Case B: }\,
Q^{adp}\supset Q_i^{opt}, \text{ for some } Q_i^{opt}\in\text{Case\,2.}
\end{cases}$$
Apply Proposition \ref{prop:optimum cover of EM}\,($ii$) and \eqref{eq:adp(F)} to $Q_i^{opt}\in\text{Case\,1}$. 
We have 
\eq{eq:mu(Sj)-3}{
\sum_{Q_i^{opt}\in\text{\,Case\,1}} h(Q_i^{opt})
\ge \frac{1}{96} \sum_{Q^{adp}\in\text{\,Case\,A}} h(Q^{adp}). 
}
On the other hand, 
\eq{eq:mu(Sj)-4}{
\sum_{Q_i^{opt}\in\text{\,Case\,2}} h(Q_i^{opt})
\ge \sum_{Q^{adp}\in\text{\,Case\,B}} \sum_{Q_i^{opt}\subset Q^{adp}} h(Q_i^{opt}). 
}
Thus, writing $\tilde\adp(Q^{adp})$ as in \eqref{eq:tilde-mu-def}, 
combining \eqref{eq:mu(Sj)-3} and \eqref{eq:mu(Sj)-4}, we have 
\eq{eq:mu(Sj)-5}{
\sum_i h(Q_i^{opt}) 
\ge \frac{1}{96}\sum_{Q^{adp}\subset \delta^{-1}.\widetilde S} \tilde\adp(Q^{adp}).
}
Let 
$\widetilde R$ 
be the rectangle concentric with $\delta^{-1}.\widetilde S$, with base $r(\widetilde S)$, and height
$$r(\widetilde R)=\delta^{-1} r(\widetilde S)-4r_j.$$ 
Then, writing $\tilde\adp(F)$ as in \eqref{eq:tilde-mu(F)-def}, 
\begin{align}
\sum_{Q^{adp}\subset \delta^{-1}.\widetilde S} \tilde\adp(Q^{adp})
&\ge\frac{1}{5} \sum_{Q^{adp}\subset \widetilde R} \tilde\adp(5Q^{adp}) \notag \\
&\ge\frac{1}{10} \sum_{Q^{adp}\subset \widetilde R}\adp(Q^{adp}), \label{eq:mu(Sj)-6} 
\end{align}
where we have used Proposition \ref{prop:optimum cover of EM}\,($i$) in the second inequality. 
To bound \eqref{eq:mu(Sj)-6} from below, 
denote by $(\tilde x, \tilde y)$ the center of $\widetilde R$. 
By the assumption that $c(\widetilde U)\in \convex(\delta.\Eb_M+l)$, 
we have 
$(\tilde x, \tilde y)\in \convex(\Eb_M).$ 
Therefore, in view of Remark \ref{rmk:Perron}\,($i$), we can find 
$(\tilde x, \hat y)\in \Eb_M$, such that 
$|\hat y-\tilde y|\le\frac{1}{2M}$. 
Let $I$ be the interval centered at $\hat y$, with 
$$|I|=\delta^{-1} r(\widetilde S)-8r_j-\frac{3}{2M}.$$
Then, arguing as in the proof of \eqref{eq:mu(5U)}, we have
\begin{align}
\sum_{Q^{adp}\subset \widetilde R}\adp(Q^{adp})
&=\#\{Q^{adp}: Q^{adp}\subset \widetilde R\}\cdot h(r_j)\notag\\
&\ge N_{\tilde x}(I)\,\frac{r(\widetilde S)}{r_j}\,h(r_j)\notag\\ 
&\ge \frac{1}{24}\cdot\frac{\delta^{-1}|\widetilde U|}{8}, 
\label{eq:mu(Sj)-7}
\end{align}
where we have used $|I|\ge \frac{1}{2}\delta^{-1} r(\widetilde S)$. 
Combining \eqref{eq:mu(Sj)-1.5}, \eqref{eq:mu(Sj)-2}, \eqref{eq:mu(Sj)-5}, \eqref{eq:mu(Sj)-6}, and \eqref{eq:mu(Sj)-7}, we obtain 
$$\mathcal M_{\delta_1}\big(\interior(\widetilde U)\cap(\delta.\Eb_{\twoM}+l);h\big) \ge \frac{1}{3}\cdot\frac{1}{960}\cdot\frac{1}{24}\cdot\frac{1}{8}\cdot |\widetilde U|.$$
This completes the proof of Proposition \ref{prop:lemma of Sj-3}.
\end{proof}


\section{Construction of $E_{\{M_n\}}$}
\label{sec:F}
As mentioned in the previous section (see \eqref{eq:13-34}), 
it suffices to consider sets $E$ which contain line segments of length at least $\frac14$ 
and with slope ranging from 0 to 1. 
More specifically, 
below we   call a compact set $E\subset\mathbb R^2$ a {Besicovitch set} if 
$$E\subset \Big[\frac12,\frac34\Big]\times\mathbb R,$$ 
$|E|=0$, and $E$ contains a line segment of the form 
$$y=ax+b(a),\quad\frac12 \le x\le \frac34$$ 
for every slope $a\in[0,1]$. 

\subsection{The construction} 
In this section we construct a family of Besicovitch\linebreak sets based on the Kakeya-type sets $\Eb_{\twoM}$ constructed in the previous section. 
The construction is based on the ideas in 
\cite{Besicovitch1928}, \cite{Sawyer1987}, \cite{Keich1999}, and \cite{Wolff1999}, 
but with some refinements. 

Given a nondecreasing sequence of positive integers $\{m_n:n\ge 1\}$, 
with $m_n\rightarrow\infty$. 
Writing $M_n = 2^{m_n}$, 
we shall use $\{\Eb_{M_n}:n\ge 1\}$ 
to inductively construct a sequence of Kakeya-type sets 
$\{\Fb_n:n\ge 1\}$ 
whose limit becomes a Besicovitch set. 
The Besicovitch set is determined by the sequence $\{M_n\}$, 
and will be denoted by $E_{\{M_n\}}$. 
The set of $E_{\{M_n\}}$, with $\{M_n\}$ as above, will be denoted by $\mathcal K$. 

To describe the construction, 
we need some notation. 
Recall from \eqref{eq:EE-def} that 
$\Eb_\n$ contains a family of line segments 
\begin{align}
\label{eq:EE-def}
\EE_\n:=\Big\{&\ell_{\alpha}(x)=\alpha x+\beta,\; 
0\le x\le 1:\notag\\ 
&\alpha=\sum_{k=1}^\n \frac{\varepsilon_k}{2^k}\;\; (\varepsilon_k=0, 1),\;\; 
\beta
=-\frac{1}{\n}\sum_{k=1}^\n \frac{(k-1)\varepsilon_k}{2^k}\Big\}.
\end{align}
For $P^j(\alpha)\subset\Pb^{(j)}_M$, 
define 
\begin{equation}
\label{eq:EM(a)-def}
\EE_M^{(\alpha)}
=\big\{\ell_{\alpha'}\in\EE_M: 
\ell_{\alpha'}\cap P^j(\alpha)\neq\varnothing\big\}.
\end{equation} 
By Proposition \ref{prop:Perron}, it is clear that 
$$\EE_M
=\bigcup_{\alpha:\, P^j(\alpha)\subset\Pb^{(j)}_M} \EE_M^{(\alpha)},\quad j=1,\cdots,M.$$
Note that, equivalently, we have (see \eqref{eq:epsilon_j}) 
\begin{equation}
\label{eq:Em(a)-range}
\EE_M^{(\alpha)}
=\{\ell_{\alpha'}\in\EE_M: 
\alpha-\epsilon_j
\le {\alpha'}< \alpha+\epsilon_j\}. \notag
\end{equation}
In particular, when $2\le j\le M$, we have 
\begin{equation}
\label{eq:EM(a)-pm}
\EE_M^{(\alpha)}=\bigcup_\pm\,\EE_M^{(\alpha\,\pm\, \epsilon_{j-1})}
\end{equation}
(note that $\epsilon_{j-1}={\epsilon_{j}}/{2}$).

\begin{definition}
\label{def:V(l)} 
Let $l(x)=ax+b\;(0\le x\le 1)$ be a line segment 
and let $\delta>0$. 
Define the \emph{triangular $\delta$-neighborhoods of} $l$ by 
\begin{align*}
\V^-_\delta(l)
&=\big\{(x,y):\;
l(x)+\delta(x - 1)\le y\le l(x),\;
0 \le x\le 1
\big\}, \\
\V^+_\delta(l)
&=\big\{(x,y):\;
l(x)\le y\le l(x)+\delta x,\; 
0 \le x\le 1
\big\}, 
\end{align*}
and\quad\quad\,\, 
$\V_\delta(l)=\V^-_\delta(l)\cup \V^+_\delta(l).
$
\end{definition}

\begin{figure}[ht]
\centering
\includegraphics[width=\linewidth]{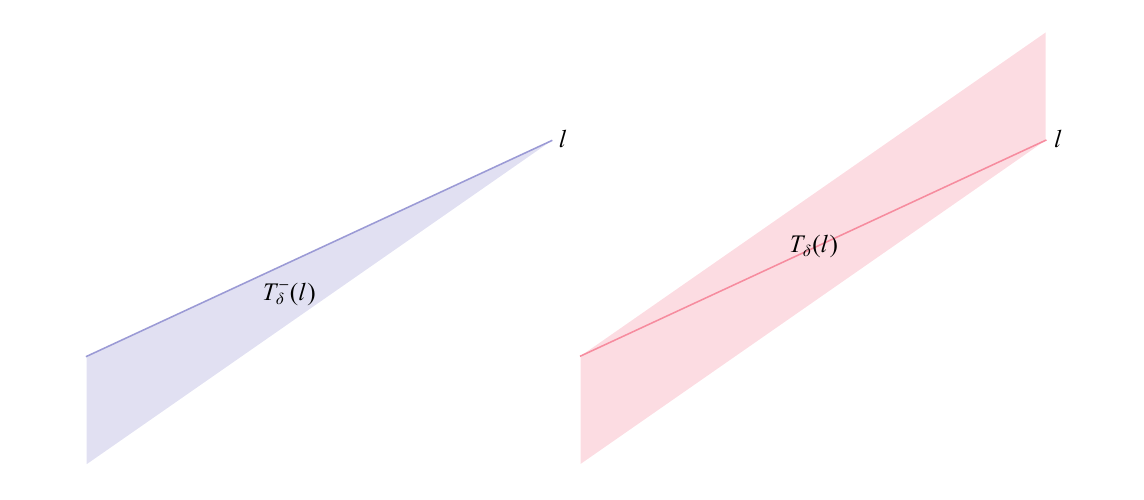}
\caption{Triangular $\delta$-neighborhoods of $l$}
\label{fig:neighborhoods}
\end{figure}
    
Note that, with $\delta=2^{-M}$, we have 
$T_i=\V^-_\delta(\ell_{\alpha_i})$ (see \eqref{eq:line-eqn}), and thus 
\eq{eq:Em=UT-}{\Eb_\n=\bigcup_{\lE\in \EE_\n} 
\V^-_{\delta}(\lE).
}
Similarly, with $\delta=2^{-M}$, we define 
\begin{equation}
\label{def:Em+}
\Eb^+_M=\bigcup_{\lE\in \EE_\n} 
\V^+_{\delta}(\lE), 
\end{equation}
and, in view of \eqref{eq:E=P-A}, define 
\begin{equation}
\label{def:Em0}
\mathring\Eb_M
=\bigcup_{\rk=1}^{\frac M2} 
\mathring\Eb^{(j)}_M
:=\bigcup_{\rk=1}^{\frac M2} 
\bigcup_{i=1}^{2^{\n-\rk}} 
\mathring P_i^{\rk} 
\end{equation}
where 
\begin{equation}
\label{eq:mathring-Pj-def}
\mathring P_i^{\rk}
:=P_i^{\rk}\cap \big(P_i^{\rk}-\delta(x-1)\big)
\end{equation}
(see \eqref{eq:E+l-def}). 
Note that $\mathring\Eb_M\not\subset\Eb_M$. 

\begin{definition}
\label{def:V(F)} 
\emph{($i$)} Let $\FF$ be a family of lines with finite slope. 
Define 
$$\V_\delta(\FF)=\bigcup_{l\in \FF}\, \V_{\delta}(l).$$ 
\emph{($ii$)} Let $\EE$ and $\FF$ be two families of lines with finite slope. 
Define 
$$\FF+\delta.\EE=\left\{l(x)+\delta\cdot\lE(x): l\in\FF,\, \lE\in\EE\right\}.$$
\end{definition}

We can now give the construction of $E_{\{M_n\}}$. 
Suppose the sequence 
$$\{M_n = 2^{m_n}: n\ge 1\}$$ 
is given. 
The Kakeya-type sets $\Fb_n$ 
take the form 
\begin{equation}
\label{def:Fn}
\Fb_n=\V_{\delta_n}(\FF_n),
\end{equation}
where $\FF_n$ is a family of lines defined inductively by 
\eq{eq:FFn-inductive}{
\FF_{n}=\FF_{n-1}+\delta_{n-1}. \EE_{M_{n}},\quad n=1,2,\cdots,
}
with $\FF_0=\{0\}$ and $\delta_n=(\text{\#}\,\FF_n)^{-1}$. 
Thus, $\delta_0=1$, $\FF_{1}=\EE_{M_1}$, 
$$\text{\#}\,\FF_n=2^{M_1}\cdots 2^{M{_{n}}},$$
and 
\eq{eq:delta_n-def}{
\delta_n=2^{-M_1}\cdots 2^{-M{_{n}}}.}
Note that the slopes contained in $\FF_n$ are given by 
$$j\delta_n=\frac{j}{\text{\#}\,\FF_n},\quad j=0, 1,\cdots, (\text{\#}\,\FF_n)-1.$$
Note also that, since $0\in\EE_M$, we have 
\begin{equation}\label{eq:FF-chain}
\FF_{0}\subset\FF_{1}\subset\cdots\subset\FF_{n}\subset\cdots.
\end{equation} 
Now define $\Fb_{n}$ by \eqref{def:Fn}, and set 
\begin{equation}
\label{eq:E=nF}
E_{\{M_n\}}=\bigcap_{n=1}^\infty \Fb_{n}
\end{equation}
(which is clearly a compact set). 
This completes the construction of $E_{\{M_n\}}$.  

\begin{figure}[ht]
  \centering
  \subfigure[$\FF_{3}^{(0)}=\frac{1}{16}.\EE_{4}$ (red)]
      {\includegraphics[width=0.49\textwidth]{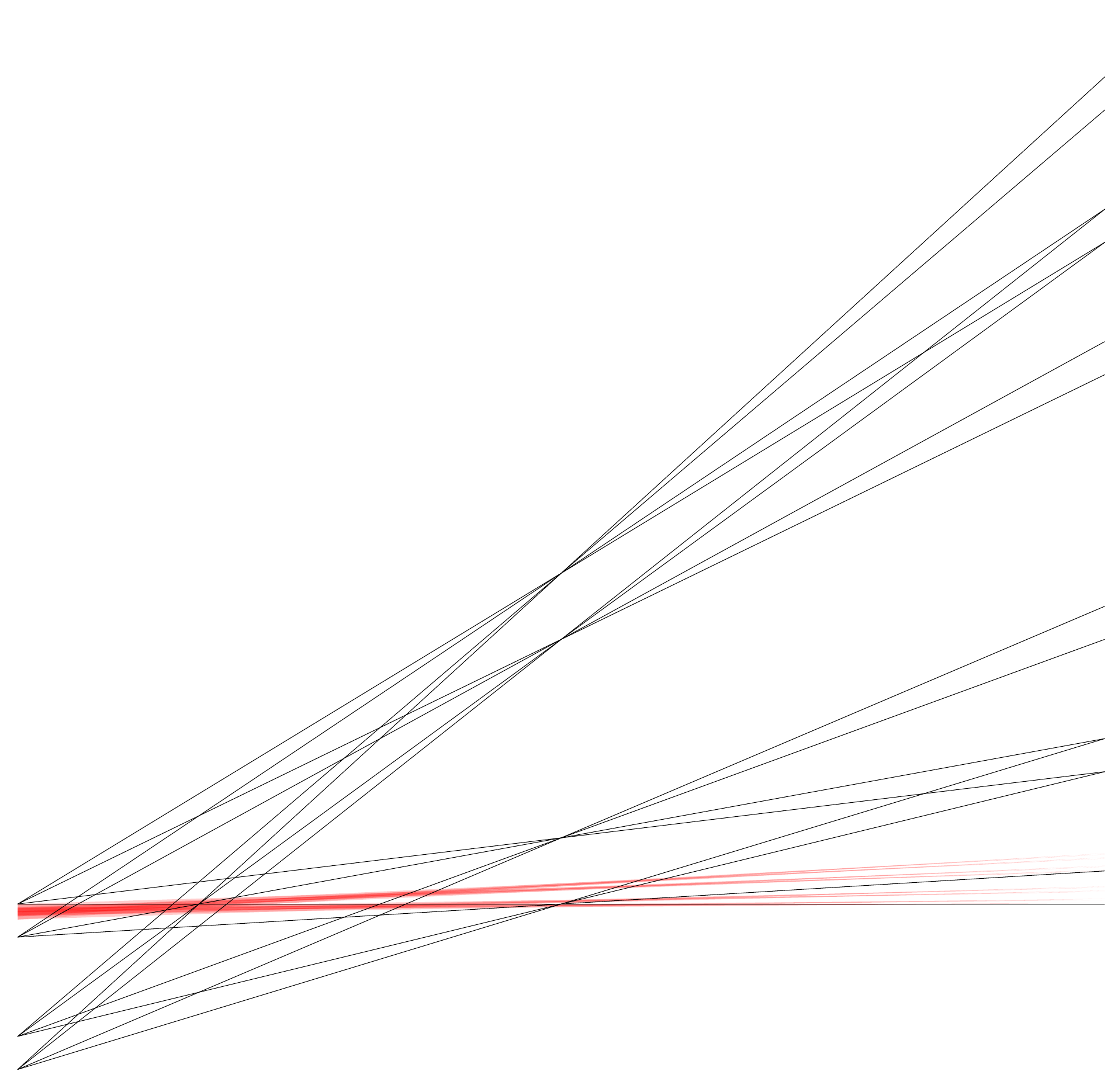}}
  \subfigure[$\FF_3=\displaystyle\bigcup_{j=0}^{15}\FF_{3}^{(j/16)}$]
      {\includegraphics[width=0.49\textwidth]{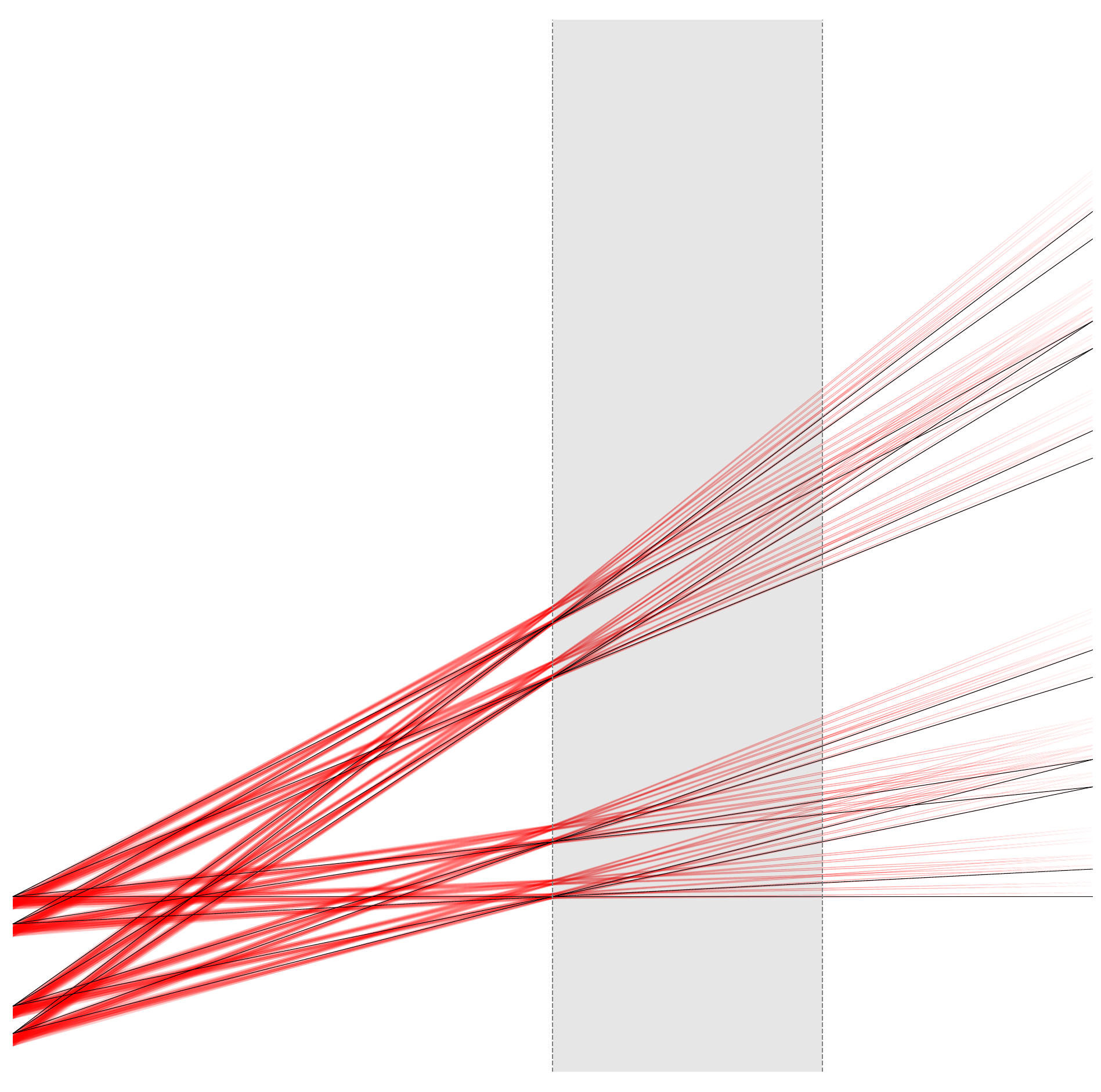}}
  \caption{$\FF_2$ (black) and $\FF_3$ (red), when $M_1=M_2=2,\, M_3=4$}
  \label{fig:F}
\end{figure}  

Note that $E_{\{M_n\}}$ is initially defined between $x=0$ and $x=1$. 
However, we shall limit its definition to $[\frac12,\frac34]\times\mathbb R$ for simplicity (see Figure \ref{fig:F}\,(b)), 
as we did with $\Eb_M$ in Section \ref{sec:E} (see \eqref{eq:13-34}). 

We make some remarks 
about the construction 
and introduce more notation. Suppose $l_{a}:=ax+b\in\FF_n$ is obtained through
$$l_{a}(x)=\sum_{j=1}^{n} \delta_{j-1}\cdot \lE_{\alpha_j}(x),$$
where $\lE_{\alpha_j}=\alpha_j x+\beta_j \in\EE_{M_j}$. 
Then 
\begin{equation}
\label{eq:b(a)-dyadic}
a=\sum_{j=1}^{n} \delta_{j-1}\alpha_j,\quad
b=\sum_{j=1}^{n} \delta_{j-1}\beta_j.
\end{equation}
In view of \eqref{eq:FF-chain}, 
this defines a function $b=b(a)$ for all dyadic rationals $a\in [0,1)$. 
When $a=1$, or when $a\in [0,1)$ is \textit{not} a dyadic rational, 
we write 
$$a=\sum_{j=1}^{\infty} \delta_{j-1}\alpha_j\quad\text{(uniquely)}$$
and define 
\begin{equation}
\label{eq:b(a)}
b(a)=\sum_{j=1}^{\infty} \delta_{j-1}\beta_j.
\end{equation}
Define also $\tilde b(0)=0$ and 
\begin{equation}
\label{eq:tilde-b(a)}
\tilde b(a)=\lim_{\tilde a\rightarrow a^-}b(\tilde a),\quad a\in(0,1].
\end{equation}
It is easy to verify that $\tilde b(a)$ 
is well-defined and 
satisfies 
$\tilde b(a)=b(a)$ 
when $a$ is {not} a dyadic rational in $(0,1)$.\footnote{\,Theorem \ref{thm:Th2}\,($iii$) below shows that $\tilde b(a)\neq b(a)$, 
when $a$ is a dyadic rational in $(0,1)$.} 
Thus, we obtain a family of lines   
\begin{align}
\label{eq:lines-la}
    \FF_{\infty}:=&\big\{l_a:=ax+b(a):a\in[0,1]\big\}\notag\\
&\cup\big\{\tilde l_a:=ax+\tilde b(a):a\in[0,1]\big\}.
\end{align}

For $l_{a}\in \FF_{n-1}$, we shall write (cf. Figure \ref{fig:F}\,(a))
$$\FF_{n}^{(a)}=\{l_{a}\} + \delta_{n-1}. \EE_{M_n}\quad\text{and}\quad \Fb_{n}^{(a)}=\V_{\delta_{n}}(\FF_{n}^{(a)}).$$
By definition, we have
\begin{equation}
\label{eq:Fn=Fn(a)}
\FF_{n}=\bigcup_{l_{a}\in \FF_{n-1}} \FF_{n}^{(a)},\quad 
\Fb_{n}=\bigcup_{l_{a}\in \FF_{n-1}} \Fb_{n}^{(a)}. 
\end{equation}
Note that, by Definition \ref{def:V(l)} and affine transformations, 
\begin{equation}
\label{eq:Fn(a)=EM}
\Fb_{n}^{(a)}=\delta_{n-1}.(\Eb_{M_n}\cup\Eb_{M_n}^+)+l_a.
\end{equation} 
We shall also consider (see \eqref{def:Em0})  
\begin{align}
\label{eq:Fn-interior}
\mathring\Fb_{n}
:=\bigcup_{l_{a}\in \FF_{n-1}}\mathring\Fb_{n}^{(a)},\quad\text{where }\;\mathring\Fb_{n}^{(a)}
:=\delta_{n-1}.\mathring\Eb_{M_n}+l_a, 
\end{align}
and 
\begin{align}
\label{eq:Fn-bdry}
\partial\Fb_{n}
:=\bigcup_{l_{a}\in\FF_{n-1}} 
\partial\Fb_{n}^{(a)}, \quad\text{where }\;
\partial\Fb_{n}^{(a)}
:=\delta_{n-1}.\bdry(\Eb_{M_n})+l_a.  
\end{align} 
Note that $\mathring \Fb_n\subset \Fb_n$,  since $\mathring\Eb_M \subset \Eb_M\cup\Eb_M^+$. 

Let $\FF$ be a family of lines with finite slope. 
When the context is clear,  
we shall often use the same notation to denote the line set 
$$\bigcup_{l\in\FF}\,\big\{(x,l(x)): 0\le x\le 1\big\}
\subset\mathbb R^2.$$ 
With this in mind, 
we define 
\begin{equation}
\label{def:Em-overline} 
\overline \Eb_M 
= \convex(\EE_M). 
\end{equation} 
Note that for 
$P_i^{\rk}=P^j(\alpha)\subset\Pb_M^{(j)}$, 
we have (see \eqref{eq:mathring-Pj-def}) 
\eq{eq:mathring-Pji-cov}{\mathring P_i^{\rk}=cov\big(\EE_M^{(\alpha)}\big),\quad 
x_j\le x\le x_{j-1},} 
or, more succinctly, 
\eq{eq:mathring-Pji-lines}{
\mathring P_i^{\rk}=cov\big(\{\ell_{\alpha-\delta},\ell_{\alpha}\}\big),} 
where $\delta=2^{-M}$ and 
$\ell_{\alpha-\delta}$ (resp. $\ell_{\alpha}$) is the line in $\EE_M^{(\alpha)}$ with slope ${\alpha-\delta}$ (resp. ${\alpha}$). 
In particular, by 
\eqref{def:Em0},  
this implies 
$\mathring\Eb_M\subset\overline \Eb_M $. 
Define also 
\begin{equation}
\label{eq:Fn-cov}
\overline\Fb_{n}=\bigcup_{l_{a}\in \FF_{n-1}}
\overline\Fb_n^{(a)},\quad\text{with }\;
\overline\Fb_n^{(a)}
:=\delta_{n-1}.\overline\Eb_{M_n}+l_a.
\end{equation} 
Note that 
\begin{equation} \label{eq:Fn(a)-cov} \overline\Fb_n^{(a)} =\convex(\FF_{n}^{(a)}). \end{equation}
By \eqref{eq:Fn-interior} and \eqref{eq:Fn=Fn(a)}, we have 
$$\mathring\Fb_{n}\subset\overline\Fb_{n}\subset\Fb_{n-1},\quad n\ge  1.$$

\subsection{Properties of $\Fb_n$ and $E_{\{M_n\}}$}
The definition of $E_{\{M_n\}}$ is clarified in the following proposition. 

\begin{proposition}
\label{proposition of Fn}
Let $\Fb_n$ and $E_{\{M_n\}}$ be as above. Then the following hold. 

\noindent\emph{(}i\emph{)} $\Fb_1\supset\Fb_2\supset\cdots\supset\Fb_n\supset\cdots$. 

\noindent\emph{(}ii\emph{)} 
$E_{\{M_n\}}=\FF_{\infty}$. In particular, $E_{\{M_n\}}$ is a line set. 

\noindent\emph{(}iii\emph{)} 
If $\{M_n\}$ is strictly increasing, then 
$$\mathcal N_{\delta_n}(E_{\{M_n\}})\cdot\delta_n^2 \log\frac{1}{\delta_n} \leq \frac{1}{2}+o(1),\quad \text{as } n\rightarrow\infty.$$
\end{proposition}

\begin{proof} 

\noindent($i$)
Write $\delta=2^{-M}$. 
By \eqref{eq:EE-def}, 
it is easy to verify by convexity 
that every $\ell\in\EE_\n$ satisfies
$$ x-1
\le \ell(x) + \delta(x-1) $$
and 
$$ \ell(x) + \delta x\le x $$
when $0\le x\le 1$.  
By definition, this implies 
\eq{eq:(i)-embed}{
\V_\delta(\EE_M)\subset \V_1(y=0).
}
After affine transformations, 
it follows that for each $l_{a}\in \FF_{n-1}$, we have 
\begin{equation}
\label{eq:Fn-in-V}
\V_{\delta_{n}}(\FF_{n}^{(a)}) \subset \V_{\delta_{n-1}} (l_{a}).
\end{equation}
Therefore, by \eqref{eq:Fn=Fn(a)} and \eqref{def:Fn}, we have 
$$\Fb_{n} 
= \displaystyle\bigcup_{l_{a}\in \FF_{n-1}} \Fb_{n}^{(a)} 
\subset \displaystyle\bigcup_{l_{a}\in \FF_{n-1}} \V_{\delta_{n-1}} (l_{a}) 
=\Fb_{n-1}.$$
This proves ($i$). 

\noindent($ii$) 
Since $\FF_n \subset \FF_{n+1}$ and $\FF_{n+k} \subset \Fb_{n+k}\, (k=1,2,\cdots$), 
by \eqref{eq:E=nF} 
we have $\FF_n\subset E_{\{M_n\}},\, n=1,2,\cdots$. 
Thus, writing 
$$\FF_{dya}=\bigcup_{n=1}^{\infty} \FF_n,$$ 
we have ${\FF_{dya}}\subset E_{\{M_n\}}$. 
By the compactness of $E_{\{M_n\}}$, 
this implies $\closure(\FF_{dya})\subset E_{\{M_n\}}$. 
On the other hand, 
it is easy to see from the definition of $\FF_{\infty}$ in \eqref{eq:lines-la} that 
$\FF_{\infty}\subset \closure(\FF_{dya})$. 
Therefore 
$$\FF_{\infty}\subset 
\closure(\FF_{dya})\subset E_{\{M_n\}}.$$

Conversely, suppose $x \in E_{\{M_n\}}$. Noting that $\V_\delta(\FF)\subset \FF(\delta)$, we have 
$$x \in \Fb_n \subset \FF_n(\delta_n)\subset
\FF_{dya}(\delta_n)
,\quad n=1,2,\cdots.$$ 
Take $n\rightarrow\infty$. This implies $x \in \closure(\FF_{dya}).$ 
Thus $E_{\{M_n\}} \subset \closure(\FF_{dya})$. It remains to show $\closure(\FF_{dya})\subset\FF_{\infty}$. 
By \eqref{eq:b(a)-dyadic}, we have $\FF_{dya} \subset \FF_{\infty}$. 
On the other hand, by \eqref{eq:b(a)}, \eqref{eq:tilde-b(a)}, and a standard compactness argument (cf. \cite[Section 3.1(3)]{Hutchinson1981}), 
$\FF_{\infty}$ forms a closed set. Therefore 
$\closure(\FF_{dya}) \subset \FF_{\infty}$. This proves ($ii$). 

\noindent($iii$)
The proof is a refinement of \cite[Lemma~1.1]{Keich1999}. 
Fix $l_{a}\in \FF_{n-1}$. 
In view of \eqref{eq:Fn(a)=EM} 
and the decomposition \eqref{eq:E=P-A}, 
each (extended) parallelogram in the $j$-th band of $\Fb_{n}^{(a)}$ has vertical width less than 
$$\delta_{n-1}\frac{1}{M_n\, 2^{M_n-j}}+\delta_n,$$
therefore can be covered by no more than 
$$\frac{\,\frac{1}{M_n}\,}{
\delta_n}\times\left(\frac{\delta_{n-1}\frac{1}{M_n\, 2^{M_n-j}}+\delta_n}{\delta_n}+3\right)
\,=\,\frac{1}{\delta_n^2\,M_n}\Big(\frac{\delta_{n-1}}{M_n\, 2^{M_n-j}}+4\delta_n\Big)$$
many $\delta_n\times\delta_n$ squares. 
Since there are $2^{M_n-j}$ many such parallelograms ($\frac{M_n}{4}<j\le \frac{M_n}{2}$), it follows that 
\begin{align*}
\mathcal N_{\delta_n} (\Fb_{n}^{(a)}) 
&\leq \displaystyle\sum_{j=\frac{M_n} {4}+1}^{\frac{M_n}{2}} 
\frac{1}{\delta_n^2\,M_n}\Big(\frac{\delta_{n-1}}{M_n}+4\delta_n 2^{M_n-j}\Big)\\
&\le \frac{1}{\delta_n^2\,M_n}\Big(\frac{\delta_{n-1}}{4}+4\delta_n 2^{3M_n/4}\Big).
\end{align*}
Thus, by \eqref{eq:E=nF} and \eqref{eq:Fn=Fn(a)},  
we obtain 
\begin{align*}
\mathcal N_{\delta_n}(E_{\{M_n\}}) 
&\leq \frac{1}{\delta_{n-1}}\cdot\frac{1}{\delta_n^2\,M_n}\Big(\frac{\delta_{n-1}}{4}+4\delta_n 2^{3M_n/4}\Big)\\
&\le \frac{1}{\delta_n^2\,M_n}\Big(\frac{1}{4}+4\cdot2^{-M_n/4}\Big), 
\end{align*}
or equivalently, 
\begin{equation}
\label{eq:N(E)}
\mathcal N_{\delta_n}(E_{\{M_n\}})\cdot\delta_n^2\,M_n
\le \frac{1}{4}+4\cdot2^{-M_n/4}.
\end{equation}
On the other hand, 
since $\{M_n\}$ is strictly increasing, we have 
\eq{eq:sum-Mn}{
M_1+M_2+\cdots+M_n\le 2M_n,
}
and so  
$\log\frac{1}{\delta_n}\le 2 M_n.$ 
Combining this with \eqref{eq:N(E)}, 
we obtain
$$\mathcal N_{\delta_n}(E_{\{M_n\}}) \cdot\delta_n^2 \log\frac{1}{\delta_n}
\leq \frac{1}{2}+8\cdot2^{-M_n/4}.$$ 
This proves ($iii$), and the proof of Proposition \ref{proposition of Fn} is complete.
\end{proof}

\begin{remark}
\label{rmk:N(bdry-Fn)}
The proof of Proposition \ref{proposition of Fn}\,(iii) 
above also gives 
\eq{eq:N(partialFn)}{
\mathcal N_{\delta_n}(\partial \Fb_n)\cdot\delta_n^2\,M_n\le 7\cdot 2^{-{M_n}/{4}},\quad n\ge 1.
}
\end{remark}

Further properties of $\Fb_n$ are summarized in the following lemma. They will be used in the proof of Theorem \ref{thm:Th1} in the next section. 
As before, we shall limit the definitions of $\Fb_n$, $\mathring \Fb_n$, and $\overline \Fb_n$ to $[\frac12,\frac34]\times\mathbb R$. 

{
\begin{lemma}
\label{lem:Fn-properties} 
For $n\ge 1$, the following hold.\\ \noindent\emph{(}i\emph{) 
(}density of $\Fb_{n+1}$ in $\Fb_{n}$\emph{)}
\eq{eq:(i)}{
\mathring\Fb_{n}\subset\overline\Fb_{n+1}.
}
\noindent\emph{(}ii\emph{) 
(}porosity of $\Fb_{n+1}$ in $\Fb_{n}$\emph{)}
$$\mathring\Fb_{n}\subset\Fb_{n+1}\Big(\frac{\delta_n}{2 M_{n+1}}\Big).$$ 
\noindent\emph{(}iii\emph{) 
(}thickness of boundary\emph{)} 
$$\Fb_n\backslash\mathring\Fb_n 
\subset 
(\partial\Fb_n)(\delta_{n}).$$
\end{lemma}
}

\begin{proof}
\noindent($i$) 
The proof is more delicate when 
$\frac12\le x\le \frac12+\frac{1}{M_n}$, 
and less so when $x\ge \frac12+\frac{1}{M_n}$. 
Below we give a unified proof. 

By \eqref{eq:Fn-interior}, we have 
\eq{eq:(i)-1}{
\mathring\Fb_{n}
=\bigcup_{l_{a}\in \FF_{n-1}} 
\big(\delta_{n-1}.\mathring\Eb_{M_n}+l_a\big). 
}
On the other hand, by \eqref{eq:Fn-cov}, 
$$
\overline\Fb_{n+1}=\bigcup_{l_{a'}\in \FF_{n}}
\big(\delta_{n}.\overline\Eb_{M_{n+1}}+l_{a'}\big). 
$$
Using \eqref{eq:FFn-inductive}, 
we can further write 
\eq{eq:(i)-2}{
\overline\Fb_{n+1}
=\bigcup_{l_{a}\in \FF_{n-1}} 
\bigcup_{\ell_{\alpha'}\in \EE_{M_{n}}} 
\left(\delta_{n}.\overline\Eb_{M_{n+1}}+\delta_{n-1}\cdot\ell_{\alpha'}+l_{a}\right). 
} 
Comparing \eqref{eq:(i)-1} and \eqref{eq:(i)-2}, 
it is clear that it suffices to show 
\eq{eq:(i)-3}{
\mathring\Eb_{M_n}\subset 
\bigcup_{\ell_{\alpha'}\in \EE_{M_{n}}} 
\big(2^{-M_{n}}.\overline\Eb_{M_{n+1}}+\ell_{\alpha'}\big). 
}
Note that this is essentially \eqref{eq:(i)} with $n=1$, since $\FF_1=\EE_{M_1}$, $\delta_1=2^{-M_1}$. 
Without loss of generality, 
we prove \eqref{eq:(i)-3} for $n=1$, 
assuming  
$$2\le M_1\le M_2,\quad 
\frac12\le x\le\frac34.$$ 

The case $M_1=2$ is relatively simple and can be verified directly. 
So we assume $M_1\ge 4$. 
For simplicity, write $\delta_2=2^{-M_2}$.  
By definition (see \eqref{def:Em0}), 
$$\mathring\Eb_{M_1} 
=\bigcup_{\rk=1}^{\frac {M_1}{2}} 
\bigcup_{i=1}^{2^{M_1-j}} 
\mathring P_i^j,$$
where $P_i^j$ ($i=1,\cdots,2^{M_1-j}$) 
are the parallelograms in $\Pb_{M_1}^{(j)}$. 
Thus, writing $P_i^j=P^j(\alpha)$,  
it suffices to show 
\eq{eq:(i)-4}{
\mathring P^j(\alpha)\subset
\bigcup_{\ell_{\alpha'}\in \EE_{M_{1}}^{(\alpha)}} 
\big(\delta_1.\overline\Eb_{M_{2}}+\ell_{\alpha'}\big),\quad 
\frac{M_1}{4}<j\le \frac{M_1}{2}. 
}
We show \eqref{eq:(i)-4} by cross-sections. 
Fix $$x\in [x_j,x_{j-1}]:=\Big[1-\frac{j}{M_1},1-\frac{j-1}{M_1}\Big].$$ 
Denote 
$$\EE_{M_{1}}^{(\alpha)}(x)=\big\{\ell_{\alpha'}(x): \ell_{\alpha'}\in\EE_{M_{1}}^{(\alpha)}\big\}\subset\mathbb R.$$ 
By \eqref{eq:mathring-Pji-cov}, we have 
$$\mathring P^{\rk}(\alpha)\big|_x=co\big(\EE_{M_{1}}^{(\alpha)}(x)\big).$$ 
On the other hand, 
\eq{eq:(i)-5}{
\bigcup_{\ell_{\alpha'}\in \EE_{M_{1}}^{(\alpha)}} 
\left(\delta_1.\overline\Eb_{M_{2}}+\ell_{\alpha'}\right)\big|_x
=\bigcup_{\ell_{\alpha'}\in \EE_{M_{1}}^{(\alpha)}} 
\Big(\delta_1\cdot\overline\Eb_{M_{2}}\big|_x+\ell_{\alpha'}(x)\Big) 
}
is a union of intervals shifted from $\delta_1\cdot\overline\Eb_{M_{2}}\big|_x$ by $\ell_{\alpha'}(x)$. 
Notice that $0\in\overline\Eb_{M_{2}}|_x$ 
as $0\in\EE_{M_{2}}$. 
Thus 
$\EE_{M_{1}}^{(\alpha)}(x)$ is contained in the union \eqref{eq:(i)-5}.  To show that the same holds true for $co\big(\EE_{M_{1}}^{(\alpha)}(x)\big)$, it reduces to showing that the union \eqref{eq:(i)-5} is connected, 
which by Definition \ref{def:epsilon-dense} is equivalent to: 
\begin{equation}
\label{eq:(i)-dense}
\EE_{M_{1}}^{(\alpha)}(x)\,\text{ is }\,\delta_1 \big|\overline\Eb_{M_{2}}|_x\big|\text{-dense}.
\end{equation}

To show \eqref{eq:(i)-dense}, we shall replace $\overline\Eb_{M_{2}}$ by its subset $\convex(\{0,\ell_{1-\delta_2}\})$ 
(where $\ell_{1-\delta_2}$ is the line segment in $\EE_{M_{2}}$ with slope $1-\delta_2$),  
and show that, in fact,  
\begin{equation}
\label{eq:(i)-dense-2}
\EE_{M_{1}}^{(\alpha)}(x)\,\text{ is }\,\delta_1 \ell_{1-\delta_2}(x)\text{-dense}.
\end{equation}
Note that, by \eqref{eq:line-eqn}, 
\eq{eq:(i)-l2-eqn}{
\ell_{1-\delta_2}(x)=(1-\delta_2)x-\frac{1}{M_2}+\frac{\delta_2}{M_2}+\delta_2.
}
So, indeed, 
$$\ell_{1-\delta_2}(x)\ge \ell_{1-\delta_2}(1/2)>0,
\quad \frac12\le x\le 1. $$
Note also that, 
by \eqref{eq:(i)-l2-eqn}, 
\eq{eq:(i)-6}{
0\le x-\ell_{1-\delta_2}(x)<\frac{1}{M_2},\quad 0\le x\le 1.
}
Therefore, in view of \eqref{eq:b>-1/M} and \eqref{eq:(i)-embed}, 
$[0,\ell_{1-\delta_2}(x)]$ is shrunk from $\overline\Eb_{M_{2}}|_x$ by less than $1/M_2$ on both sides. 

Now by \eqref{eq:EM(a)-pm}, we can write 
\eq{eq:(i)-branching}{
\EE_{M_1}^{(\alpha)}(x)
=\bigcup_\pm\,\EE_{M_1}^{(\alpha\pm\epsilon_{j-1})}(x),\quad \text{with}\quad\epsilon_{j-1}=\frac{1}{2^{M_1-j+2}}.
}
By Proposition \ref{prop:Perron}, $\EE_{M_1}^{(\alpha\pm\epsilon_{j-1})}(x)$ 
consists of the upper endpoints of the intervals (see \eqref{eq:ExPxTx} and \eqref{eq:Em=UT-})
$$\big\{T^x_i: T_i\cap P^{j-1}(\alpha\pm\epsilon_{j-1})\neq\varnothing\big\},$$
where $T_i$ are the triangles in $ \Eb_{M_1}$. 
By Remark \ref{rmk:eta-connected}, 
these intervals form a $\frac{\delta_1}{M_1}$-connected family. 
Thus, by Definition \ref{def:eta-connected}, 
\begin{equation}
\label{eq:(i)-dense2}
\EE_{M_1}^{(\alpha\pm\epsilon_{j-1})}(x)
\,\text{ is }\,
\Big(\delta_1(1-x)-\frac{\delta_1}{M_1}\Big)\text{-dense}
\end{equation}
(note that $|T^x_i|=\delta_1(1-x)$). 
On the other hand, by \eqref{eq:(i)-6}, 
\begin{align*}
\delta_1 \ell_{1-\delta_2}(x)
&>\delta_1 \Big(x-\frac{1}{M_2}\Big)\\
&\ge \delta_1 (1-x)-\frac{\delta_1}{M_1},\quad\quad
\frac12\le x\le 1.
\end{align*}
Combining this with \eqref{eq:(i)-dense2}, 
we see that both $\EE_{M_1}^{(\alpha-\epsilon_{j-1})}(x)$ 
and $\EE_{M_1}^{(\alpha+\epsilon_{j-1})}(x)$ are $\delta_1 \ell_{1-\delta_2}(x)$-dense. 

To obtain the same conclusion for $\EE_{M_1}^{(\alpha)}(x)$, 
we use \eqref{eq:(i)-branching} and Lemma \ref{lem:dense-union}. 
More specifically, 
consider the lines $$\ell_{\alpha-\delta_1},\, 
\ell_{\alpha-\epsilon_{j-1}} 
\in \EE_{M_1}^{(\alpha-\epsilon_{j-1})},$$
and 
$$\ell_{\alpha+\epsilon_{j-1}-\delta_1},\, 
\ell_{\alpha}
\in \EE_{M_1}^{(\alpha+\epsilon_{j-1})}.$$
By \eqref{eq:mathring-Pji-lines}, we have $$\ell_{\alpha-\delta_1}(x)
\le \ell_{\alpha-\epsilon_{j-1}}(x),\, 
\ell_{\alpha+\epsilon_{j-1}-\delta_1}(x)
\le \ell_{\alpha}(x), 
\quad x_j\le x\le x_{j-1}.$$ 
Thus, by 
Lemma \ref{lem:dense-union}, 
it suffices to show that the middle lines satisfy 
\eq{eq:(i)-7}{
\ell_{\alpha+\epsilon_{j-1}-\delta_1}(x)
-\ell_{\alpha-\epsilon_{j-1}}(x)
\le \delta_1 \ell_{1-\delta_2}(x),\quad x_j\le x\le x_{j-1}. 
}
Notice that 
the slopes of the lines on the two sides of \eqref{eq:(i)-7} 
satisfy 
$$(\alpha+\epsilon_{j-1}-\delta_1)
-(\alpha-\epsilon_{j-1})>\delta_1(1-\delta_2).$$ 
So it suffices to show that \eqref{eq:(i)-7}  holds at $x=x_{j-1}$. 
Indeed, the lines 
$\ell_{\alpha+\epsilon_{j-1}-\delta_1}$ and 
$\ell_{\alpha-\epsilon_{j-1}}$ 
can be used to find
the vertical base of the triangle $A^j_i\subset\Pb_{M_1}$ at $x=x_{j-1}$ (see Figure \ref{fig:E}). 
Together with Proposition \ref{prop:EM-observe}\,($iv$), 
this yields  
\eq{eq:(i)-8}{
\ell_{\alpha+\epsilon_{j-1}-\delta_1}(x_{j-1})
-\ell_{\alpha-\epsilon_{j-1}}(x_{j-1}) 
= \frac{\delta_1}{M_1} + \delta_1(1-x_{j-1}). 
}
On the other hand, using $2\le M_1\le M_2$, 
it is easy to verify that 
\eq{eq:(i)-9}{
\frac{1}{M_1}+(1-x)<\ell_{1-\delta_2}(x), 
\quad \frac{1}{2}+\frac{1}{M_1}\le x\le 1. 
}
In particular, 
taking $x=x_{j-1}$, \eqref{eq:(i)-9} gives 
\eq{eq:(i)-10}{
\frac{\delta_1}{M_1}+\delta_1(1-x_{j-1})
<\delta_1 \ell_{1-\delta_2}(x_{j-1}). 
}
Combining \eqref{eq:(i)-8} and \eqref{eq:(i)-10} now shows that \eqref{eq:(i)-7} holds at $x=x_{j-1}$. 
This completes the proof of ($i$). 

\noindent($ii$) 
By ($i$), it suffices to show 
$$\overline\Fb_{n+1}\subset \Fb_{n+1}\Big(\frac{\delta_n}{2 M_{n+1}}\Big).$$
By Remark \ref{rmk:Perron}\,($i$), 
the vertical gaps in $\delta_{n}. \Eb_{M_{n+1}}$ 
are of size less than $\frac{\delta_n}{M_{n+1}}$. 
Therefore, by \eqref{eq:Fn(a)=EM}, 
the same holds true for $\Fb_{n+1}^{(a)}$, for each $l_{a}\in \FF_{n}$. 
Since $\FF_{n+1}^{(a)}\subset \Fb_{n+1}^{(a)}$, 
it follows that 
$$
cov\big(\FF_{n+1}^{(a)}\big)\subset \Fb_{n+1}^{(a)}\Big(\frac{\delta_n}{2 M_{n+1}}\Big).$$ 
Taking union over $l_a$, by \eqref{eq:Fn(a)-cov}, 
\eqref{eq:Fn-cov}, and \eqref{eq:Fn=Fn(a)}, we obtain  
$$\overline\Fb_{n+1}
=\bigcup_{l_{a}\in \FF_{n}} \overline\Fb_{n+1}^{(a)}
\subset \bigcup_{l_{a}\in \FF_{n}} \Fb_{n+1}^{(a)}\Big(\frac{\delta_n}{2 M_{n+1}}\Big)
=\Fb_{n+1}\Big(\frac{\delta_n}{2 M_{n+1}}\Big).$$
This proves ($ii$). 

\noindent($iii$)
Write $\delta=2^{-M}$. 
Using the decomposition \eqref{eq:E=P-A}, 
it is easy to verify that
\begin{equation*}
\label{eq:sym-difference}
(\Eb_M\cup \Eb^+_M)
\backslash \mathring\Eb_M 
\subset (\bdry(\Eb_M))(\delta), \quad \frac12 \le x\le \frac34.
\end{equation*}
By \eqref{eq:Fn(a)=EM}, \eqref{eq:Fn-interior}, 
and \eqref{eq:Fn-bdry}, 
this implies 
$$\Fb_n^{(a)}\backslash\mathring\Fb_n^{(a)} 
\subset \big(\partial\Fb_n^{(a)}\big)(\delta_{n}),\quad l_{a}\in \FF_{n-1}.$$  
Taking union over $l_{a}$, 
by \eqref{eq:Fn=Fn(a)}, \eqref{eq:Fn-interior}, 
and \eqref{eq:Fn-bdry}, we obtain 
$$\Fb_n\backslash\mathring\Fb_n 
\subset (\partial\Fb_n)(\delta_{n}).$$
This proves ($iii$), and the proof of Lemma \ref{lem:Fn-properties} is complete. 
\end{proof}


\section{Generalized Hausdorff dimension of $E_{\{M_n\}}$}
\label{sec:Th1}

In this section we prove Theorem A. It is an immediate consequence of the following slightly more general result, upon taking 
$\phi(r)=\log\log\log(1/r)$ (for example). 
Note that it is implied by Proposition \ref{proposition of Fn}\,($iii$) that $E_{\{M_n\}}$ has 
finite Hausdorff measure with respect to 
$h(r)=r^2\,{\log}(1/r)$. 

\begin{theorem}
\label{thm:Th1}
Let $\phi: (0,1]\rightarrow (0,\infty)$ be a decreasing, continuous function 
which satisfies 
$\lim_{r\rightarrow0}\phi(r)=\infty$ 
and 
\eq{eq:phi-cond}{
\phi(r)\le \log(1/r),\quad 0<r\le 1. 
}
Then $E_{\{M_n\}}$ has positive 
Hausdorff measure with respect to the gauge 
$$h_1(r):=r^2\, {\log}(1/r)\,\phi(r),$$
provided that 
\begin{equation}
\label{eq:M-condition-2}
\phi\big(\frac{1}{M_{n+1}}\big) \ge \frac{1}{c_1} \log\big(\frac{1}{\delta_n}\big)\,\phi(\delta_n) 
\end{equation}
holds for sufficiently large $n$ 
(where $c_1$ is as in Proposition \ref{prop:lemma of Sj-3}). 
\end{theorem} 

Since the Hausdorff measure is monotonic with respect to the gauge, 
condition \eqref{eq:phi-cond} is not essential and is only for regularizing the growth of $\phi(r)$ as $r\rightarrow 0$. 
Note however that, 
since $\delta_n\approx 2^{-M_n}$ 
(see \eqref{eq:delta_n-def} and \eqref{eq:sum-Mn}), 
condition \eqref{eq:M-condition-2} imposes a strong growth condition on the sequence $\{M_n\}$. 

Recall from Section \ref{sec:F} 
that $E_{\{M_n\}}=\bigcap_{n\ge 1} \Fb_{n}$. 
The proof of Theorem \ref{thm:Th1} proceeds by showing that 
the $[\delta_{n+1},1]$-limited $h_1$-measure of $\Fb_{n+1}$ only decreases slightly from the $[\delta_{n},1]$-limited $h_1$-measure of $\Fb_{n}$, 
when $M_{n+1}$ is large enough that 
\eqref{eq:M-condition-2} holds. 
Since by construction 
$$\FF_{n+1}=\FF_{n}+\delta_{n}.\EE_{M_{n+1}},$$
the proof relies heavily on the properties of {$\Eb_{M}$} studied in Section \ref{sec:E}. 

Fix a sufficiently large integer $n_0\ge 1$. 
Below we assume that 
$$M_1<\cdots<M_{n}$$ 
($n\ge n_0$) are given,  
and denote $M_{n+1}$ by $M$. 

We start with a local lower bound for the $h_1$-measure of $\Fb_{n+1}$ at scale $\frac1M$. 

\begin{lemma}
\label{lemma of Sj}
Suppose $M=M_{n+1}$ satisfies 
\begin{equation}
\label{eq:M-condition-1}
\frac1M<\delta_n. 
\end{equation}
Then, for any dyadic square $\widetilde U$ with 
$r(\widetilde U)=\frac1M$
and $c(\widetilde U)\in
\overline\Fb_{n+1}$, we have 
\begin{equation}\label{tilde-Sj-bound}
\mathcal{M}^{{\frac 1M}}_{\delta_{n+1}}(\interior(\widetilde U)\cap \Fb_{n+1};h_1) 
\ge c_1|\widetilde U|\,\phi\big(\frac{1}{M}\big), 
\end{equation}
where $c_1$ is as in Proposition \ref{prop:lemma of Sj-3}.
\end{lemma}

\begin{proof}
Since $c(\widetilde U)\in
\overline\Fb_{n+1}$, by \eqref{eq:Fn-cov}, 
there exists $l_a\in\FF_n$ such that 
$$c(\widetilde U)\in \convex(\delta_n.\Eb_M+l_a).$$
Fix such an $l_a$. Applying Proposition \ref{prop:lemma of Sj-3} with $\delta=\delta_n$, $\delta_1=\delta_{n+1}$, and $l=l_a$, 
we have 
\eq{eq:lem5.1-1}{
\mathcal M_{\delta_{n+1}}\big(\interior(\widetilde U)\cap(\delta_n.\Eb_{M}+l_a);h\big)
\ge c_1 |\widetilde U|.
}
Since (see \eqref{eq:Fn=Fn(a)}, \eqref{eq:Fn(a)=EM})
$$\delta_n.\Eb_{M}+l_a\subset \Fb_{n+1},$$ 
it follows from \eqref{eq:lem5.1-1} that 
\begin{align*}
\mathcal{M}^{{\frac 1M}}_{\delta_{n+1}}(\interior(\widetilde U)\cap \Fb_{n+1};h_1)
&\ge \mathcal{M}^{{\frac 1M}}_{\delta_{n+1}}(\interior(\widetilde U)\cap \Fb_{n+1};h)\cdot \phi\big(\frac{1}{M}\big)\\
&\ge \mathcal M_{\delta_{n+1}}\big(\interior(\widetilde U)\cap(\delta_n.\Eb_{M}+l_a);h\big)\cdot \phi\big(\frac{1}{M}\big)\\
&\ge c_1|\widetilde U|\,\phi\big(\frac{1}{M}\big), 
\end{align*}
where we have used the monotonicity of $\phi$ in the first inequality, 
and the monotonicity of the scale-limited measure in the second inequality. 
This completes the proof of Lemma \ref{lemma of Sj}.  
\end{proof}

\begin{remark}
$(i)$ A more careful analysis of the contribution from different scales (similar to the proof of Proposition \ref{prop:lemma of Sj-3}) shows that the factor $\phi\big(\frac{1}{M}\big)$ in \eqref{tilde-Sj-bound} can be improved to $\phi\big(2^{-M}\big)$, 
provided that $\phi$ satisfies
$$\phi(r^2)\le \phi(r)+C,\quad 0<r\le 1$$
for an absolute constant $C$. 

\noindent$(ii)$ The proof above also shows\,\footnote{\,In what follows, 
all scale-limited measures are with respect to $h_1$. So we write $\mathcal M(E)=\mathcal M(E;h_1)$ for simplicity.} 
\eq{eq:rmk5.2-1}{
\mathcal{M}_{\delta_{n+1}}^{\frac{1}{M}}(\delta_n. \Eb_{M}+l_a) 
\gtrsim
\delta_n\,\phi\big(\frac{1}{M}\big),\quad l_a\in\FF_n. 
}
In comparison, covering $\delta_n. \Eb_{M}+l_a$ by $\delta_n\times\delta_n$ squares gives
\eq{eq:rmk5.2-2}{
\mathcal{M}_{\delta_{n}}(\delta_n. \Eb_{M}+l_a) 
\lesssim \delta_n \log\big(\frac{1}{\delta_n}\big)\,\phi(\delta_n). 
}
Note that \eqref{eq:rmk5.2-2} is strictly smaller than \eqref{eq:rmk5.2-1} when $\frac1M\ll \delta_n$. 
\end{remark}

Next, we apply Lemma \ref{lemma of Sj} to establish a local lower bound for the $h_1$-measure of $\Fb_{n+1}$ at scale $\delta_n$. Note that it is implied by condition \eqref{eq:M-condition-2} that 
$\frac{1}{M}<\delta_n$. 

\begin{lemma}
\label{coro of M}
Suppose 
$M=M_{n+1}$ satisfies \eqref{eq:M-condition-2}. 
Then, for any dyadic square $\widetilde Q\subset\mathring\Fb_{n}$ with $r(\widetilde Q)=\delta_n$, we have
$$\mathcal{M}_{\delta_{n+1}}^{\delta_n}
(\interior(\widetilde Q) \cap \Fb_{n+1}) 
=
h_1(\widetilde Q).$$ 
\end{lemma}

\begin{proof}
By \eqref{eq:M(Q)}, it is clear that 
\eq{eq:M(Q)-upper}{
\mathcal{M}_{\delta_{n+1}}^{\delta_n}
(\interior(\widetilde Q) \cap \Fb_{n+1}) 
\le \mathcal{M}_{\delta_{n+1}}^{\delta_n}
(\widetilde Q)
=h_1(\widetilde Q).
} 
To show the lower bound, let 
$\{U_k\}\cup \{Q_i\}$ be an optimal $[\delta_{n+1},\delta_n]$-cover of $\interior(\widetilde Q) \cap \Fb_{n+1}$ with respect to $h_1$, where 
$$
\delta_{n+1}\le r(U_k)\le\frac1M<r(Q_i)\le\delta_n.$$
Now decompose 
\eq{eq:M(Q)-0}{
\widetilde Q=\Big(\bigcup_i\, \widetilde U_i\Big) \bigcup \Big(\bigcup_i\, Q_i\Big),
}
where each $\widetilde U_i$ is a dyadic square with $r(\widetilde U_i)=\frac1M$. 
Notice that 
$$\big\{U_k:\, U_k\subset\widetilde U_i\big\}$$ 
forms a $[\delta_{n+1},\frac1M]$-cover of $\interior(\widetilde U_i)\cap \Fb_{n+1}$. 
Thus, by Definition \ref{def:h-measures}, 
\eq{eq:M(Q)-2}{
\sum_{k:\,U_k\subset\widetilde U_i} h_1(U_k)\ge \mathcal{M}^{{\frac 1M}}_{\delta_{n+1}}(\interior(\widetilde U_i)\cap \Fb_{n+1}). 
}
On the other hand, 
by Lemma \ref{lem:Fn-properties}\,($i$) and the assumption $\widetilde Q\subset\mathring\Fb_{n}$, 
we have $\widetilde Q\subset\overline\Fb_{n+1}$. 
Therefore, we can apply Lemma \ref{lemma of Sj} to each $\widetilde U_i$ to obtain 
\eq{eq:M(Q)-1}{
\mathcal{M}^{{\frac 1M}}_{\delta_{n+1}}(\interior(\widetilde U_i)\cap \Fb_{n+1}) 
\ge c_1|\widetilde U_i|\,\phi\big(\frac{1}{M}\big). 
} 
Combining \eqref{eq:M(Q)-2} and \eqref{eq:M(Q)-1}, we see that 
\begin{align}
\mathcal{M}^{\delta_n}_{\delta_{n+1}}(\interior(\widetilde Q)\cap \Fb_{n+1}) \notag
&=\sum_{k} h_1(U_k) 
+ \sum_{i} h_1(Q_i) \notag\\
&=\sum_{i}\sum_{k:\,U_k\subset\widetilde U_i} h_1(U_k) 
+ \sum_{i} h_1(Q_i) \notag\\
&\ge c_1\sum_i |\widetilde U_i|\,\phi\big(\frac1M\big) 
+ \sum_i |Q_i| \,\varphi\big(r(Q_i)\big), \label{eq:M(Q)-3}
\end{align}
where $\varphi(r)=\log(\frac{1}{r})\,\phi(r)$. 
Applying condition \eqref{eq:M-condition-2} to the first sum, and the monotonicity of $\varphi$ to the second sum (noting that $r(Q_i)\le\delta_n$), we have 
\begin{align}
\eqref{eq:M(Q)-3}
&\ge 
\Big(\sum_i |\widetilde U_i|
+ \sum_i |Q_i|\Big)\,\varphi(\delta_n)
\notag\\
&= |\widetilde Q| \,\varphi(\delta_n)\notag\\
&= h_1(\widetilde Q), 
\label{eq:M(Q)-4}
\end{align}
where we have used \eqref{eq:M(Q)-0} in the second line. 
Combining \eqref{eq:M(Q)-upper}, \eqref{eq:M(Q)-3}, and \eqref{eq:M(Q)-4}, this completes the proof of Lemma \ref{coro of M}. 
\end{proof}

\begin{remark}
A similar argument using 
Lemma \ref{lem:Fn-properties}\,($ii$) and Lemma \ref{lem:invisible}
shows that, under the same condition, 
$$\mathcal{M}_{\delta_{n+1}}^{r(\widetilde Q)}
(\interior(\widetilde Q) \cap \Fb_{n+1}) 
=
h_1(\widetilde Q)$$ 
holds for any dyadic square $\widetilde Q\subset\mathring\Fb_{n}$ with $r(\widetilde Q)\ge\delta_n$. 
\end{remark} 

The following lemma gives a preliminary lower bound for $\mathcal{M}_{\delta_{n+1}}(\Fb_{n+1})$. 
It will be used together with Lemma \ref{lem:Mn} to derive a uniform lower bound for $\mathcal{M}_{\delta_n}(\Fb_{n})$ 
(see Corollary \ref{cor:M(F)>0} below). 

\begin{lemma}
\label{lem:prelim-lower-bound}
Suppose $M=M_{n+1}$ satisfies \eqref{eq:M-condition-1}. Then we have 
$$\mathcal{M}_{\delta_{n+1}}(\Fb_{n+1})>\frac{\phi(1)}{8}\frac{1}{M^2}.$$
\end{lemma}

\begin{proof}
By Definition \ref{def:h-measures} 
and the monotonicity of $\phi$, 
we have 
\eq{eq:prelim-1}{
\mathcal{M}_{\delta_{n+1}}(\Fb_{n+1})
\ge \phi(1)|\Fb_{n+1}|. 
}
In view of \eqref{eq:Fn=Fn(a)} and \eqref{eq:Fn(a)=EM}, 
\begin{align}
|\Fb_{n+1}|
&\ge |\delta_n. \Eb_M + l_0|\notag\\
&=\delta_n |\Eb_M|\notag\\
&>\frac1M |\Eb_M| 
\label{eq:prelim-2}
\end{align}
(where we have used condition \eqref{eq:M-condition-1} in the last line). 
On the other hand, 
it is easy to verify using \eqref{eq:E=P-A} that 
\eq{eq:prelim-3}{
|\Eb_M|>\frac{1}{8M}. 
}
Thus, combining 
\eqref{eq:prelim-1}, \eqref{eq:prelim-2}, and \eqref{eq:prelim-3}, we obtain 
$$\mathcal{M}_{\delta_{n+1}}(\Fb_{n+1})>\frac{\phi(1)}{8}\frac{1}{M^2}.$$
This completes the proof of Lemma \ref{lem:prelim-lower-bound}.
\end{proof}

\begin{remark}
Alternatively, one can use the arguments in \cite[Theorem~2]{Keich1999} to show that, unconditionally, $|\Fb_{n+1}|\approx \frac{1}{M}$. 
\end{remark} 

We now compare the $[\delta_{n+1},1]$-limited $h_1$-measure of $\Fb_{n+1}$ with the $[\delta_{n},1]$-limited $h_1$-measure of $\Fb_{n}$. 

\begin{lemma}
\label{lem:Mn}
Suppose $M=M_{n+1}$ satisfies \eqref{eq:M-condition-2}. Then we have 
\begin{equation}\label{eq:iteration}
\mathcal{M}_{\delta_{n+1}}(\Fb_{n+1})
\ge 
\mathcal{M}_{\delta_n}(\Fb_{n}) - C {M_{n}}\, {2^{-{M_n}/{4}}}, 
\end{equation}
where $C=196$. 
\end{lemma}

\begin{proof}

\begin{figure}[htb]
\centering
\subfigure[$\,\mathring\Fb_{n}\subset\overline\Fb_{n+1}\subset\Fb_n$\quad]
{\includegraphics[width=0.28\textwidth]{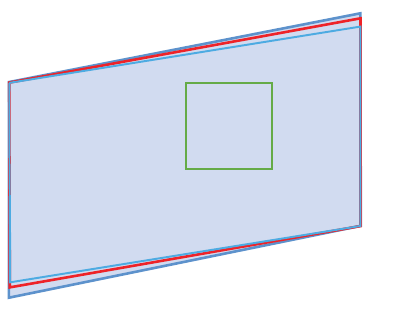}}
\subfigure[$\,\partial\Fb_{n}$]
{\includegraphics[width=0.28\textwidth]{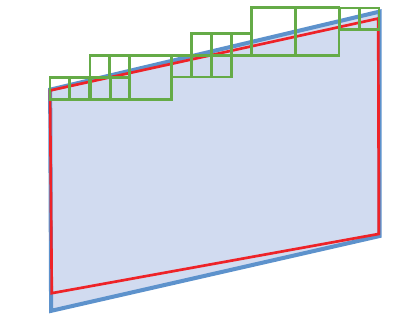}}
\subfigure[larger dyadic squares (green) covering $\Fb_{n}$]{\includegraphics[width=0.59\textwidth]{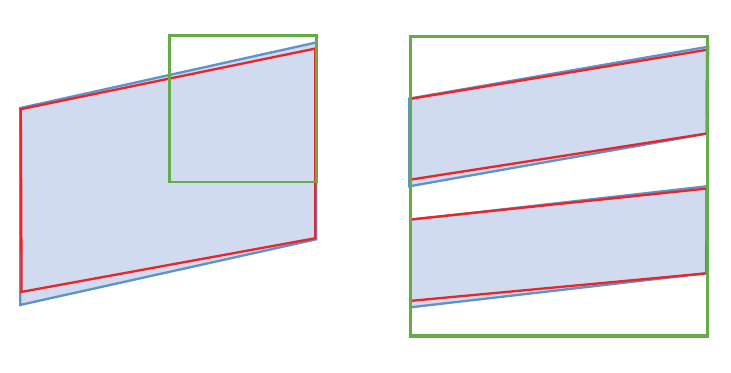}}
\includegraphics[width=0.5\textwidth]{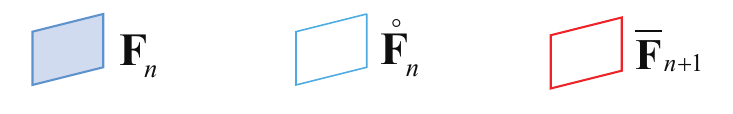}
\caption{An illustration of the proof of Lemma \ref{lem:Mn}}
\label{fig:Qk}
\end{figure}  
Let ${\mathcal{Q}}_n$ be 
the collection of dyadic squares 
$\widetilde Q\subset\mathring\Fb_{n}$ 
with $r(\widetilde Q)=\delta_n$. 
Set 
$$\overset{\mathbin{\rule{.6mm}{.6mm}}}{\Fb}_n
=\bigcup_{\widetilde Q\in {\mathcal{Q}}_n} \widetilde Q.$$
By monotonicity, we have 
\eq{eq:M(Fn)-M(Fn+1)-0}{\mathcal{M}_{\delta_{n+1}}(\Fb_{n+1})\ge \mathcal{M}_{\delta_{n+1}}(\overset{\mathbin{\rule{.6mm}{.6mm}}}{\Fb}_n\cap\Fb_{n+1}).}
As in the proof of Lemma \ref{coro of M}, 
let $\{U_k\}\cup \{Q_i\}$ 
be an optimal $\delta_{n+1}$-cover of $\overset{\mathbin{\rule{.6mm}{.6mm}}}{\Fb}_n\cap\Fb_{n+1}$ with respect to $h_1$, where 
$$\delta_{n+1}\le r(U_k)\le\delta_n<r(Q_i)\le1.$$ 
Without loss of generality, assume 
$\{U_k\}\prec {\mathcal{Q}}_n$ 
(see Remark \ref{rmk:def-2.1}). 
Let 
$$\big\{\widetilde Q_i\big\}=\big\{\widetilde Q\in {\mathcal{Q}}_n: 
\exists\, k \text{ such that } U_k\subset \widetilde Q\big\}.$$ 
Then for each $\widetilde Q_i$, 
$$\big\{U_k:\, U_k\subset \widetilde Q_i\big\}$$
forms a $[\delta_{n+1},\delta_{n}]$-cover of 
$\interior(\widetilde Q_i)\cap \Fb_{n+1}$. 
Therefore, by Lemma \ref{coro of M}, 
\begin{align}
\sum_{k:\,U_k\subset \widetilde Q_i} h_1(U_k)
\ge \mathcal{M}_{\delta_{n+1}}^{\delta_n}
(\interior(\widetilde Q_i) \cap \Fb_{n+1}) 
= h_1(\widetilde Q_i).
\label{eq:M(Fn)-M(Fn+1)-1}
\end{align}
On the other hand, using Lemma \ref{coro of M} again, we have 
$${\mathcal{Q}}_n\prec \{\widetilde Q_i\}\cup\{Q_i\}.$$ 
Thus 
$\{\widetilde Q_i\}\cup\{Q_i\}$ 
forms a $[\delta_n,1]$-cover of $\overset{\mathbin{\rule{.6mm}{.6mm}}}{\Fb}_n$, 
and so 
\begin{align}
\sum_{i} h_1(\widetilde Q_i)+\sum_{i} h_1(Q_i)
&\ge \mathcal{M}_{\delta_{n}}
(\overset{\mathbin{\rule{.6mm}{.6mm}}}{\Fb}_n) \notag\\
&\ge \mathcal{M}_{\delta_{n}}
(\Fb_n)
-\mathcal{M}_{\delta_{n}}
(\Fb_n\backslash\overset{\mathbin{\rule{.6mm}{.6mm}}}{\Fb}_n), 
\label{eq:M(Fn)-M(Fn+1)-2}
\end{align}
where we have used the subadditivity of the scale-limited measure in the last line. 
Combining  \eqref{eq:M(Fn)-M(Fn+1)-0}, 
\eqref{eq:M(Fn)-M(Fn+1)-1}, and \eqref{eq:M(Fn)-M(Fn+1)-2}, 
we see that 
\begin{align}
\mathcal{M}_{\delta_{n+1}}
(\Fb_{n+1})
&\ge \mathcal{M}_{\delta_{n+1}}(\overset{\mathbin{\rule{.6mm}{.6mm}}}{\Fb}_n\cap\Fb_{n+1})\notag\\
&=\sum_{k} h_1(U_k) 
+ \sum_{i} h_1(Q_i) \notag\\
&=\sum_{i}\sum_{k:\,U_k\subset\widetilde Q_i} h_1(U_k) 
+ \sum_{i} h_1(Q_i) \notag\\
&\ge \sum_{i} h_1(\widetilde Q_i)+\sum_{i} h_1(Q_i) \notag\\ 
&\ge  \mathcal{M}_{\delta_{n}}
(\Fb_n)
-\mathcal{M}_{\delta_{n}}
(\Fb_n\backslash\overset{\mathbin{\rule{.6mm}{.6mm}}}{\Fb}_n). \notag
\end{align}
It remains to show 
\eq{eq:M(Fn)-M(Fn+1)-3}{
\mathcal{M}_{\delta_{n}}
(\Fb_n\backslash\overset{\mathbin{\rule{.6mm}{.6mm}}}{\Fb}_n)
\le 196\, M_n\,2^{-M_n/4}. 
}
Indeed, by \eqref{eq:Fn-interior}, \eqref{def:Em0}, and \eqref{eq:Fn-bdry}, we have 
\begin{align*}
\mathring\Fb_n\backslash\overset{\mathbin{\rule{.6mm}{.6mm}}}{\Fb}_n
&= \bigcup_{l_{a}\in \FF_{n-1}} (\delta_{n-1}.\mathring\Eb_{M_n}+l_a)
\backslash \overset{\mathbin{\rule{.6mm}{.6mm}}}{\Fb}_n\\
&\subset \bigcup_{l_{a}\in\FF_{n-1}} 
\big(\delta_{n-1}.\bdry(\Eb_{M_n})+l_a\big)(3\delta_n)\\
&=(\partial\Fb_n)(3\delta_n). 
\end{align*}
Together with Proposition \ref{lem:Fn-properties}\,($iii$), 
this implies  
\begin{align}
\Fb_n\backslash\overset{\mathbin{\rule{.6mm}{.6mm}}}{\Fb}_n
=(\Fb_n\backslash\mathring\Fb_n)\cup(\mathring\Fb_n\backslash\overset{\mathbin{\rule{.6mm}{.6mm}}}{\Fb}_n) \subset (\partial\Fb_n)(3\delta_n).
\label{eq:M(Fn)-M(Fn+1)-4}
\end{align}
On the other hand, by Remark \ref{rmk:N(bdry-Fn)}, we have 
\begin{align}
\mathcal N_{\delta_n}\big((\partial\Fb_n)(3\delta_n)\big)
\le 7 \,\mathcal N_{\delta_n}(\partial\Fb_n)
\le \frac{49}{\delta_n^2\,M_n}\,2^{-M_n/4}. \label{eq:M(Fn)-M(Fn+1)-5}
\end{align}
Thus, combining \eqref{eq:M(Fn)-M(Fn+1)-4} and \eqref{eq:M(Fn)-M(Fn+1)-5}, 
we obtain 
\begin{align}
\mathcal{M}_{\delta_{n}}
(\Fb_n\backslash\overset{\mathbin{\rule{.6mm}{.6mm}}}{\Fb}_n)
&\le 
\mathcal N_{\delta_n}(\Fb_n\backslash\overset{\mathbin{\rule{.6mm}{.6mm}}}{\Fb}_n)\cdot h_1(\delta_n)\notag\\
&\le 
\mathcal N_{\delta_n}\big((\partial\Fb_n)(3\delta_n)\big)\cdot \delta_n^2 \Big(\log\frac{1}{\delta_n}\Big)\phi(\delta_n) \notag\\
&\le \frac{49}{M_n}\Big(\log\frac{1}{\delta_n}\Big)^2 \,2^{-M_n/4} \notag\\ 
&\le 49\cdot 4\, M_n\,2^{-M_n/4}, \notag
\end{align}
where we have used condition \eqref{eq:phi-cond} in the third line, and $\log\frac{1}{\delta_n}\le 2M_n$ in the last line (see \eqref{eq:sum-Mn}). 
This proves \eqref{eq:M(Fn)-M(Fn+1)-3}, 
and the proof of Lemma \ref{lem:Mn} is complete. 
\end{proof}

\begin{remark}
\label{rmk:prove-by-nesting}
Alternatively, one can use Lemma \ref{lem:nesting} to keep track of how the squares in an optimal $\delta_{n}$-cover of 
$\Fb_n$ break down into smaller squares to form an optimal $\delta_{n+1}$-cover of 
$\Fb_{n+1}$, and establish an estimate of the form
$$
\mathcal{M}_{\delta_{n+1}}(\Fb_{n+1}) \ge 
\left(1-\eta_n\right) \mathcal{M}_{\delta_n}(\Fb_{n})$$
for a rapidly decreasing sequence $\eta_n$. 
\end{remark}

Combining Lemma \ref{lem:prelim-lower-bound} and Lemma \ref{lem:Mn}, we obtain the following. 

\begin{corollary}
\label{cor:M(F)>0}
Suppose \eqref{eq:M-condition-2} holds for all $n\ge n_0$. Then we have 
\eq{eq:infM(Fn)-0}{
\inf_{n\ge 1}\mathcal M_{\delta_n}(\Fb_n)>0.
}
\end{corollary}

\begin{proof}
Let $n_1>n_0$ be a sufficiently large integer so that
\eq{eq:infM(Fn)-1}{
2\,C {M_{n_1}} {2^{-M_{n_1}/4}} < \frac{\phi(1)}{8}\frac{1}{M_{n_1}^2}, 
}
where $C$ is as in Lemma \ref{lem:Mn}. 
Then, iterating \eqref{eq:iteration}, we have 
\eq{eq:infM(Fn)-2}{
\mathcal{M}_{\delta_{n}}(\Fb_{n})
\ge 
\mathcal{M}_{\delta_{n_1}}(\Fb_{n_1}) - 2\,C {M_{n_1}}{2^{-{M_{n_1}}/{4}}},\quad n\ge n_1. 
}
Apply Lemma \ref{lem:prelim-lower-bound} with $n+1=n_1$. We have 
\eq{eq:infM(Fn)-3}{
\mathcal{M}_{\delta_{n_1}}(\Fb_{n_1})
\ge 
\frac{\phi(1)}{8}\frac{1}{M_{n_1}^2}. 
}
Thus, combining \eqref{eq:infM(Fn)-2}, \eqref{eq:infM(Fn)-3}, and \eqref{eq:infM(Fn)-1}, 
we obtain 
$$\inf_{n\ge n_1}\mathcal M_{\delta_n}(\Fb_n)
>\frac{\phi(1)}{8}\frac{1}{M_{n_1}^2}-2\,C{M_{n_1}}{2^{-M_{n_1}/4}}
>0.$$
On the other hand, by \eqref{eq:prelim-1}, 
we have 
$$\mathcal M_{\delta_n}(\Fb_n)\ge \phi(1)|\Fb_n|>0,\quad  n=1,\cdots,n_1-1.$$ 
Combining the last two bounds shows \eqref{eq:infM(Fn)-0}, 
and the proof of Corollary \ref{cor:M(F)>0} is complete. 
\end{proof}

Theorem \ref{thm:Th1} is now an immediate consequence of Corollary \ref{cor:M(F)>0} and the following proposition. 

\begin{proposition}
\label{prop:equiv-positive}
$E_{\{M_n\}}$ has positive 
Hausdorff measure with respect to $h_1$ if and only if \eqref{eq:infM(Fn)-0} holds. 
\end{proposition}
\begin{proof}
Recall that the Hausdorff measure of $E\subset\mathbb R^2$ with respect to a gauge function $h$ is defined by (cf. \cite[\S\,2.5]{Falconer2003})
\begin{align*}
\mathcal H^h(E)
=&\lim_{\delta\rightarrow 0} 
\mathcal H^h_\delta(E)\\
:=&\lim_{\delta\rightarrow 0} 
\Big(\inf_{\{S_i\}}\sum_i h(S_i) \Big), 
\end{align*}
where the infimum is taken over all possible covers of $E$ by open squares $S_i$ (not necessarily dyadic) with sidelength $r(S_i)\le \delta$. 

Write $E=E_{\{M_n\}}$. 
Since $E$ is compact, 
by standard compactness and doubling arguments (cf. \cite[\S\,2.4]{Falconer2003}), we have 
$$\mathcal H^{h_1}(E)>0\iff \lim_{\delta\rightarrow 0} \mathcal M_{\delta}(E)>0.$$
By monotonicity, the last inequality is equivalent to 
\eq{eq:equiv-positive-1}{
\inf_{n\ge 1} \mathcal M_{\delta_n}(E)>0. 
} 
Since $E\subset \Fb_n$, we have 
$\mathcal M_{\delta_n}(E)\le \mathcal M_{\delta_n}(\Fb_n)$. Thus \eqref{eq:equiv-positive-1} $\Rightarrow$ \eqref{eq:infM(Fn)-0} holds. 

To show \eqref{eq:infM(Fn)-0} $\Rightarrow$ \eqref{eq:equiv-positive-1}, 
it suffices to show 
\eq{eq:equiv-positive-2}{
\mathcal M_{\delta_n}(\Fb_n) \le 3\,\mathcal M_{\delta_n}(E). 
} 
Indeed, 
let $\{Q_i\}$ be an optimal $\delta_n$-cover of $E$ with respect to $h_1$. 
Then, since
$$\Fb_n\subset\FF_n(\delta_n)\subset E(\delta_n)$$
(see the proof of Proposition \ref{proposition of Fn}\,($ii$)), we have 
$$\Fb_n\subset \bigcup_i\, 3Q_i,$$ 
where $3Q_i$ is as in \eqref{eq:(2k+1)Q}. 
By Definition \ref{def:h-measures}, it follows that 
$$M_{\delta_n}(\Fb_n) 
\le 3\sum_i h_1(Q_i)
=3\,\mathcal M_{\delta_n}(E).$$ 
This proves \eqref{eq:equiv-positive-2}, and the proof of Proposition \ref{prop:equiv-positive} is complete. 
\end{proof}

\begin{remark}
\label{rmk:open-problems}
The proof of Theorem \ref{thm:Th1} above does not seem to apply to the Hausdorff measure $\mathcal H^h(E_{\{M_n\}})$ with 
$h(r)=r^2\log(1/r)$. 
In particular, the following questions remain open: 

\noindent$(i)$ Is it true that $\mathcal H^h(E_{\{M_n\}})>0$ when $\{M_n\}$ grows sufficiently fast? 

\noindent$(ii)$ Is it true that $\mathcal H^h(E_{\{2^n\}})=0$?  
\end{remark} 


\section{Minimality of $E_{\{M_n\}}$}
\label{sec:Th2} 

In this section we prove Theorem B. 
By rotation, it follows immediately from the following more precise statements. 

\begin{theorem}
\label{thm:Th2}
For every $E_{\{M_n\}}\in\mathcal K$, the following statements hold. 

\noindent $(i)$ If $a>1\,$or $a<0$, then every line $l\,$with slope $a$ satisfies $\mathcal H^1(l\cap E_{\{M_n\}})=0$. 

\noindent $(ii)$ If $a=0,\, 1$, or $a\in (0,1)$ is not a dyadic rational, then
$l_a=ax+b(a)$ is the only line 
with slope $a$, such that 
$\mathcal H^1(l_a\cap E_{\{M_n\}})>0$.\footnote{\,Here, and below, $l_a$ and $\tilde l_a$ refer to the lines defined in \eqref{eq:lines-la}.}   

\noindent $(iii)$ If $a\in (0,1)$ is a dyadic rational, then 
there are exactly two lines $l=l_a$ and $l=\tilde l_a$ with slope $a$, 
such that 
$\mathcal H^1(l\cap E_{\{M_n\}})>0$. 

\end{theorem} 

\begin{figure}[ht]
\centering
\includegraphics[width=\linewidth]{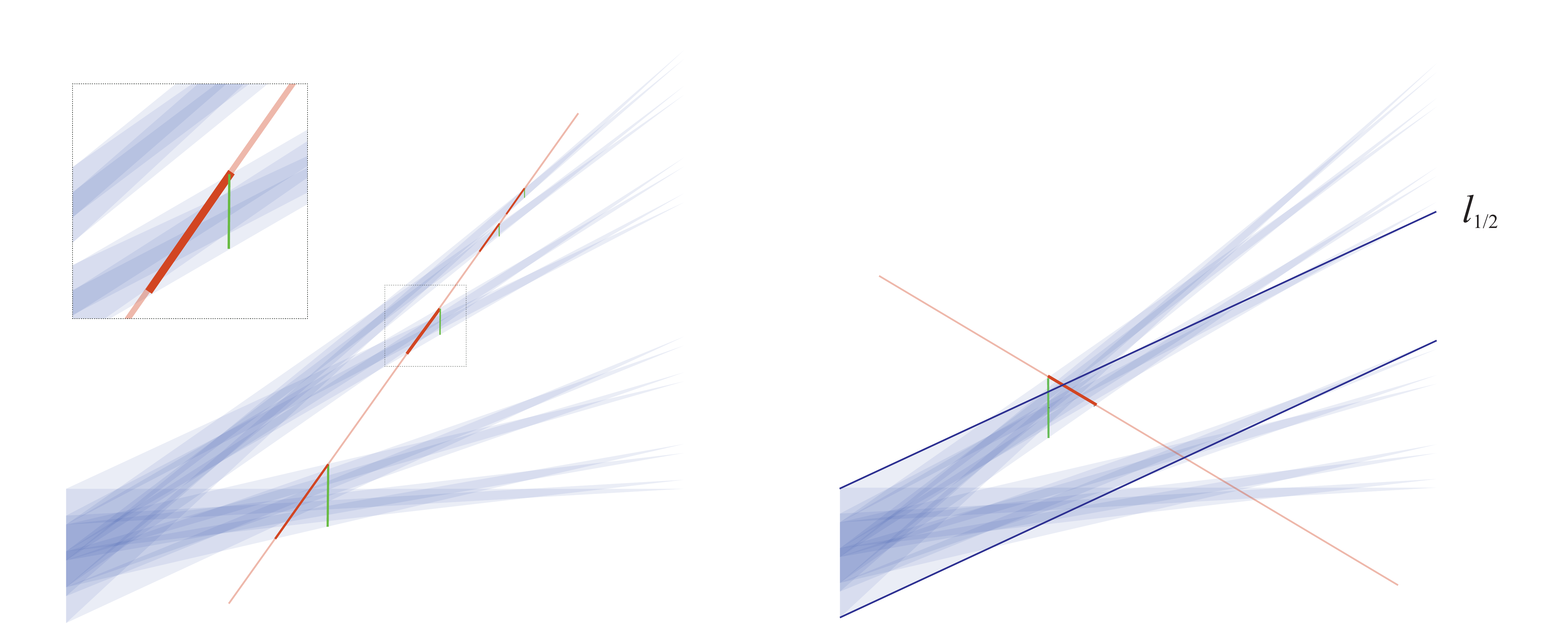}
\caption{Every line with slope $a>1$ or $a<0$ 
has vanishing intersection with $\Eb_M$, as $M\rightarrow\infty$}
\label{fig: minimal}
\end{figure}

The proof of Theorem \ref{thm:Th2} is based on the following proposition, which can be regarded as an extension of \eqref{eq:1/N}. 

\begin{proposition}
\label{prop:prop1-for-Th2}
For every line $l=ax+b$  with $|a|\ge 2$, we have 
$$
\mathcal H^1(l\cap \Eb_\n)\le \frac{2\sqrt{5}}{M}. $$
\end{proposition}

\begin{proof} 
By \eqref{eq:Em-in-Pm}, we have 
$$\Eb_\n
\subset\Pb_M
=\bigcup_{\rk=1}^M\bigcup_{i=1}^{2^{\n-\rk}}P_i^{\rk}.$$
Therefore,  
\begin{equation}
\label{eq:H1(lcapEM)}
\mathcal H^1(l\cap \Eb_\n)
\le \sum_{\rk=1}^M\sum_{i=1}^{2^{\n-\rk}} \mathcal H^1(l\cap P_i^{\rk}).
\end{equation}
Since $P_i^{\rk}$ has vertical width less than $r_j$ (see \eqref{eq:rj-eqn}) and slope in $[0,1]$, and since $|a|\ge 2$, we have 
$$\mathcal H^1(l\cap P_i^{\rk})
\le \sqrt{5}\, r_j.$$
Thus, by \eqref{eq:H1(lcapEM)}, 
\begin{equation}
\label{eq:H1(lcapEM)-2}
\mathcal H^1(l\cap \Eb_\n)
\le \sqrt{5}\, \sum_{\rk=1}^M\sum_{i:\,l\cap P_i^{\rk}\neq\varnothing} r_j.
\end{equation}
Slide the corresponding vertical intervals (of length $r_j$) along the Perron tree $\Pb_M$ to $x=0$. 
The slid intervals form dyadic subintervals of $[-1/M,0]$, overlapping up to multiplicity 2 (due to the condition $|a|\ge 2$). From this and \eqref{eq:H1(lcapEM)-2} it follows that 
$$\mathcal H^1(l\cap \Eb_\n)
\le \frac{2\sqrt{5}}{M}.$$
This completes the proof of Proposition \ref{prop:prop1-for-Th2}. 
\end{proof}

In order to apply Proposition \ref{prop:prop1-for-Th2}, we shall need the following extension. 

\begin{proposition}
\label{prop:prop2-for-Th2}
Let $l=ax+b$ and $l_0=a_0x+b_0$ be two lines with $|a-a_0|\ge 2\delta$ and $a_0\in[0,1]$. Then we have 
\begin{equation}\label{eq:prop2-for-Th2}
\mathcal H^1\big(l\cap (\delta.\Eb_\n+l_0)\big)
\le \Big(1+\frac{1}{|a-a_0|}\Big)\cdot\delta\cdot\frac{10}{M}, 
\end{equation}
\end{proposition}

\begin{proof}
Since $|a_0|\le 1$, 
applying the affine transformation $y\mapsto y+a_0x+b$ to covers of $(l-l_0)\cap \delta.\Eb_\n$, we obtain 
\begin{equation}
\label{eq:root-5}
\mathcal H^1\big(l\cap (\delta.\Eb_\n+l_0)\big)
\le \sqrt{5}\,\mathcal H^1\big((l-l_0)\cap \delta.\Eb_\n\big).
\end{equation}
Note that $l-l_0$ has slope $a-a_0$. 
Since $|a-a_0|\ge 2\delta$, 
applying another affine transformation $y\mapsto \delta y$, we obtain 
\begin{equation}
\label{eq:l-l_0} 
\mathcal H^1\big((l-l_0)\cap \delta.\Eb_\n\big)
\le \Big(1+\frac{1}{|a-a_0|}\Big)\cdot\delta\cdot\mathcal H^1\big(\delta^{-1}(l-l_0)\cap \Eb_\n\big).
\end{equation}
Now by Proposition \ref{prop:prop1-for-Th2}, we have 
$$\mathcal H^1\big(\delta^{-1}(l-l_0)\cap \Eb_\n\big)
\le \frac{2\sqrt 5}{M}.$$
Combining this with \eqref{eq:root-5} and \eqref{eq:l-l_0} yields \eqref{eq:prop2-for-Th2}, as desired. 
\end{proof}

We are now ready to prove Theorem \ref{thm:Th2}. 

\begin{proof}[Proof of Theorem \ref{thm:Th2}]

\noindent($i$) 
By \eqref{eq:E=nF}, we have 
$$\mathcal H^1(l\cap E_{\{M_n\}})\le \mathcal H^1(l\cap \Fb_{n+1}),\quad n=0, 1, \cdots.$$ 
Let $n_0\ge 1$ be the smallest index such that $\min(|a-1|,|a|) \ge 2\delta_{n_0}$. Fix $n\ge n_0$. 
By \eqref{eq:Fn=Fn(a)}, 
\begin{equation}
\label{eq:Fn+1=UFn}
\Fb_{n+1}
=\bigcup_{0\le j<\delta_n^{-1}} \Fb_{n+1}^{(j\delta_{n})}. 
\end{equation}
Using \eqref{eq:Fn(a)=EM}, we can write 
\begin{equation}
\label{eq:V=EMpm}
\Fb_{n+1}^{(j\delta_{n})}
=\big(\delta_{n}.\Eb_{M_{n+1}} \cup \delta_{n}.\Eb_{M_{n+1}}^+\big)
+l_{j\delta_{n}}, 
\end{equation} 
where $l_{j\delta_{n}}$ is the line in $\FF_{n}$ with slope $j\delta_{n}$. 
Since we are restricting to $\frac12\le x\le \frac 34$, it is easy to verify by \eqref{def:Em+} that 
$\Eb_M^+$ can be covered by three vertical shifts of $\Eb_M$. 
More precisely, here we have 
$$\delta_{n}.\Eb_{M_{n+1}}^+
\subset \bigcup_{k=1}^3 
\Big(\delta_{n}.\Eb_{M_{n+1}} + k\frac{\delta_{n+1}}{4}\Big), $$
Therefore, by \eqref{eq:V=EMpm}, 
\begin{equation}
\label{eq:F(j)n+1}
\Fb_{n+1}^{(j\delta_{n})}
\subset \bigcup_{k=0}^3 
\Big(\delta_{n}.\Eb_{M_{n+1}} + l_{j\delta_{n}} + k\frac{\delta_{n+1}}{4}\Big).
\end{equation}
Since 
$$|a-j\delta_{n}|\ge \min(|a-1|,|a|)\ge 2 \delta_{n_0}\ge2 \delta_{n},$$ 
by Proposition \ref{prop:prop2-for-Th2}, we have for $0\le k\le 3$,   
\begin{align*}
\mathcal H^1\left(l\cap \Big(\delta_{n}.\Eb_{M_{n+1}} + l_{j\delta_{n}} + k\frac{\delta_{n+1}}{4}\Big)\right)
&\le \Big(1+\frac{1}{2\delta_{n_0}}\Big)\cdot\delta_n\cdot\frac{10}{M_{n+1}}\\
&\le \frac{\delta_n}{\delta_{n_0}}\frac{10}{M_{n+1}}. 
\end{align*}
Thus, by \eqref{eq:F(j)n+1}, 
$$\mathcal H^1(l\cap \Fb_{n+1}^{(j\delta_{n})})
\le \frac{\delta_n}{\delta_{n_0}}\frac{40}{M_{n+1}}.$$
Summing over $j$, in view of  \eqref{eq:Fn+1=UFn}, we obtain 
$$\mathcal H^1(l\cap \Fb_{n+1})
\le \frac{1}{\delta_{n_0}}\frac{40}{M_{n+1}}.$$
Taking $n\rightarrow\infty$, 
since $M_n\rightarrow\infty$, 
we see that 
$$\mathcal H^1(l\cap E_{\{M_n\}})
\le \lim_{n\rightarrow\infty}\mathcal H^1(l\cap \Fb_{n+1})
=0.$$
This proves ($i$). 

\noindent($ii$) For simplicity, we consider only the case where $a\in(0,1)$ is not a dyadic rational 
(the cases $a=0$ and $a=1$ are similar; see also the proof of \noindent($iii$) below). 
By Proposition \ref{proposition of Fn}\,($ii$), 
we have $\mathcal H^1 (l \cap E_{\{M_n\}})>0$ when $l=l_a$. 
So we need to show $\mathcal H^1 (l \cap E_{\{M_n\}})=0$ when $l\neq l_a$. 

Fix such a line $l=ax+b$ 
with $b\neq b(a)$. 
Let $n_0\ge 1$ be the smallest index such that $|b-b(a)|>\delta_{n_0}$, 
and let $j_0$ be the integer such that 
$$a\in \big(j_0\delta_{n_0},(j_0+1)\delta_{n_0}\big).$$
Applying \eqref{eq:Fn-in-V} inductively, 
we have for any $n\ge n_0$, 
\begin{equation}
\label{eq:inclusion}
\bigcup_{j\delta_n\in [j_0\delta_{n_0},(j_0+1)\delta_{n_0})} \Fb_{n+1}^{(j\delta_{n})}
\subset \V_{\delta_{n_0}}(l_{j_0\delta_{n_0}}).    
\end{equation}
In particular, 
restricting to $0\le x\le 1$, 
we have 
$l_a \subset \V_{\delta_{n_0}}(l_{j_0\delta_{n_0}})$, 
which in turn implies 
\begin{equation}
\label{eq:emptyset}
l\cap \V_{\delta_{n_0}}(l_{j_0\delta_{n_0}})=\varnothing, 
\end{equation}
since $|b-b(a)|>\delta_{n_0}$. 

Now let $n_1\ge n_0$ be the smallest index such that 
$$\min\big(a-j_0\delta_{n_0},(j_0+1)\delta_{n_0}-a\big)
\ge 2\delta_{n_1}.$$ 
Fix $n\ge n_1$. 
Denote 
$$\Fb_{n+1}'
=\bigcup_{j\delta_n\notin [j_0\delta_{n_0},(j_0+1)\delta_{n_0})} \Fb_{n+1}^{(j\delta_{n})}.$$
By \eqref{eq:inclusion} and \eqref{eq:emptyset}, we have 
$$l\cap \Fb_{n+1}=l\cap \Fb_{n+1}'.$$ 
On the other hand, arguing as in the proof of ($i$) above, 
we have  
$$\mathcal H^1(l\cap \Fb_{n+1}')
\le \frac{1}{\delta_{n_1}}\frac{40}{M_{n+1}}.$$ 
Taking $n\rightarrow\infty$, 
we see that      
$$\mathcal H^1(l\cap E_{\{M_n\}})
\le \lim_{n\rightarrow\infty}
\mathcal H^1(l\cap \Fb_{n+1}')=0.$$
This proves ($ii$). 

\noindent($iii$) 
First, we show that 
$\mathcal H^1 (l \cap E_{\{M_n\}})=0$ when $l\neq l_a,\, \tilde l_a$ 
(that is, $l=ax+b$ 
with $b\neq b(a),\,\widetilde b(a)$). 
Let $n_0$ be the smallest index such that 
$l_a\in \FF_{n_0}$, 
and let $n_1\ge n_0$ be the smallest index such that 
\begin{equation}
\label{eq:n_1-cond}
\min\big(|b-b(a)|,|b-\widetilde b(a)|\big)>\delta_{n_1}.
\end{equation}
Then, for any $n\ge n_1$, 
we have 
\begin{equation}
\label{eq:inclusions}
\bigcup_{j\delta_n\in [a,a+\delta_{n_1})} \Fb_{n+1}^{(j\delta_{n})}
\subset \V_{\delta_{n_1}}(l_a),\quad\bigcup_{j\delta_n\in [a-\delta_{n_1},a)} \Fb_{n+1}^{(j\delta_{n})}
\subset \V_{\delta_{n_1}}(l_{a-\delta_{n_1}}).
\end{equation}
Thus, restricting to $0\le x\le 1$, we have  
$l_a \subset \V_{\delta_{n_1}}(l_a)$ (trivially) and  
$\tilde l_a \subset \V_{\delta_{n_1}}(l_{a-\delta_{n_1}})$ (by \eqref{eq:Fn-in-V}). 
It follows from \eqref{eq:n_1-cond} that 
\begin{equation}
\label{eq:emptysets}
l\cap \V_{\delta_{n_1}}(l_a)=\varnothing,\quad 
l\cap \V_{\delta_{n_1}}(l_{a-\delta_{n_1}})=\varnothing.
\end{equation}
Now fix $n>n_1$. 
Denote 
$$\Fb_{n+1}'
=\bigcup_{j\delta_n\notin [a-\delta_{n_1},a+\delta_{n_1})} \Fb_{n+1}^{(j\delta_{n})}.$$
By \eqref{eq:inclusions} and \eqref{eq:emptysets}, we have 
$l\cap \Fb_{n+1}=l\cap \Fb_{n+1}'.$ 
Arguing as in the proof of ($ii$) above (note that $\delta_{n_1}\ge 2\delta_n$), 
we have 
$$\mathcal H^1(l\cap \Fb_{n+1}')
\le \frac{2}{\delta_{n_1}}\frac{40}{M_{n+1}}.$$ 
Taking $n\rightarrow\infty$, 
we obtain
$$\mathcal H^1(l\cap E_{\{M_n\}})
\le \lim_{n\rightarrow\infty}
\mathcal H^1(l\cap \Fb_{n+1}')=0,$$ 
as desired. 

Next, we show that indeed $l_a\neq \tilde l_a$ 
(that is, $b(a)\neq \widetilde b(a)$). 
The proof is based on showing
$|l_a(0)-l_{a-\delta_{n_0}}(0)|\neq |\tilde l_a(0)-l_{a-\delta_{n_0}}(0)|$ and relies on an arithmetic observation. 
More precisely, since $M$ is an {even} number, applying \eqref{eq:intercept-difference} to 
$\delta_{n_0-1}.\Eb_{M_{n_0}}+l_{a_0}$ 
(suppose $l_a\in \FF_{n_0}^{(a_0)})$, 
we have 
\begin{equation}
\label{eq:intercept-lower-bound}
|l_{a}(0)-l_{a-\delta_{n_0}}(0)|\ge \delta_{n_0-1}\frac{1}{M_{n_0}\,2^{M_{n_0}}}=\frac{\delta_{n_0}}{M_{n_0}}.
\end{equation}
On the other hand, 
applying 
\eqref{eq:b>-1/M} 
to $\delta_{n_0}.\Eb_{M_{n_0+1}}+l_{a-\delta_{n_0}}$, we have 
\begin{equation}
\label{eq:intercept-upper-bound}
-\frac{\delta_{n_0}}{M_{n_0+1}}<\tilde l_{a}(0)-l_{a-\delta_{n_0}}(0)\le 0. 
\end{equation}
Since ${M_{n_0+1}}\ge {M_{n_0}}$, it follows from 
\eqref{eq:intercept-lower-bound} and \eqref{eq:intercept-upper-bound} that 
 $$|l_a(0)-l_{a-\delta_{n_0}}(0)|
 >|\tilde l_a(0)-l_{a-\delta_{n_0}}(0)|.$$
Thus $l_{a}(0)\neq \tilde l_{a}(0).$ 
This proves ($iii$), 
and the proof of Theorem \ref{thm:Th2} is complete. 
\end{proof}

\noindent
\textbf{Acknowledgment.} 
This work was supported in part by the National Key R$\&$D Program of China 
(No.\;2022YFA1005700) 
and the NNSF of China 
(No.\;12371105). 
The authors would like to thank Sanghyuk Lee for helpful comments on an earlier version of the manuscript.

\nocite{Lee2013}
\bibliographystyle{abbrv}
\bibliography{bibliography}

\end{document}